\documentclass{elsarticle}
\usepackage{lineno,hyperref}
\usepackage{amsmath,amssymb,amsthm}
\usepackage{bm}
\usepackage{color}
\usepackage{amscd}
\usepackage[all]{xy}
\usepackage{tikz}
\usepackage{bm}
\modulolinenumbers[5]

\journal{Journal of \LaTeX\ Templates}

\newtheorem{thm}{Theorem}[section]
\newtheorem{defi}[thm]{Definition}
\newtheorem{rem}[thm]{Remark}
\newtheorem{prop}[thm]{Proposition}
\newtheorem{cor}[thm]{Corollary}
\newtheorem{lem}[thm]{Lemma}
\newtheorem{exam}[thm]{Example}

\newtheorem{notation}[thm]{Notation}

\newcommand{\GL}{{\rm GL}}
\newcommand{\gl}{{\rm gl}}
\newcommand{\Gal}{{\rm Gal}}
\newcommand{\Hom}{{\rm Hom}}
\newcommand{\stab}{{\rm stab}}
\newcommand{\id}{{\rm id}}

\newcommand{\Seq}{{\rm Seq}}
\newcommand{\diag}{{\rm diag}}
\newcommand{\sdgal}{\sigma\delta\mbox{-}\Gal}
\newcommand{\SDgal}{\Sigma\Delta\mbox{-}\Gal}
\newcommand{\dgal}{\delta\mbox{-}\Gal}
\newcommand{\sgal}{\sigma\mbox{-}\Gal}

\newcommand{\Q}{{\mathbb{Q}}}
\newcommand{\Z}{{\mathbb{Z} }}
\newcommand{\K}{{\mathcal{K} }}
\newcommand{\F}{{\mathcal{F} }}
\newcommand{\calZ}{{\mathcal{Z} }}
\newcommand{\calX}{{\mathcal{X} }}

\newcommand{\calS}{{\mathcal{S} }}
\newcommand{\R}{{\mathcal{R}}}

\newcommand{\calD}{{\mathcal{D} }}
\newcommand{\I}{{\mathcal{I}}}

\newcommand{\lc}{{\rm\mbox{lc}}}

\newcommand{\frakm}{\mathbf{m}}
\newcommand{\frakn}{\mathbf{n}}
\newcommand{\bfc}{\mathbf{c}}
\newcommand{\bfb}{\mathbf{b}}

\newcommand{\frakq}{\mathfrak{q}}

\newcommand{\fraki}{\mathfrak{i}}
\newcommand{\frakh}{\mathfrak{h}}

\newcommand{\ie}{{\it i.e.\,\,}}

\begin{document}

\begin{frontmatter}
\title{Galois Groups of Linear Difference-Differential Equations}
\tnotetext[mytitlenote]{This work was supported by NSFC under Grants No.11771433 and No.11688101, Beijing Natural Science Foundation (Z190004) and  National Key Research and Development Project 2020YFA0712300.}
\author{Ruyong Feng and Wei Lu}

\address{KLMM,Academy of Mathematics and Systems Science, Chinese Academy of Sciences and School of Mathematics, University of Chinese Academy of Sciences, 100190, Beijing\\
ryfeng@amss.ac.cn, luwei18@mails.ucas.ac.cn} 


\begin{abstract}
We study the relation between the Galois group $G$ of a linear difference-differential system and two classes $\mathcal{C}_1$ and $\mathcal{C}_2$ of groups that are the Galois groups of the specializations of the linear difference equation and the linear differential equation in this system respectively. We show that almost all groups in $\mathcal{C}_1\cup \mathcal{C}_2$ are algebraic subgroups of $G$, and there is a nonempty subset of $\mathcal{C}_1$ and a nonempty subset of $\mathcal{C}_2$ such that $G$ is the product of any pair of groups from these two subsets. These results have potential application to the computation of the Galois group of a linear difference-differential system. We also give a criterion for testing linear dependence of elements in a simple difference-differential ring, which generalizes Kolchin's criterion for partial differential fields.
\end{abstract}

\begin{keyword}
\texttt{Linear difference-differential equations, Galois groups, Specializations, simple difference-differential rings}
\MSC[2010] 12H05 \sep 12H10 \sep 39A06 \sep 34A30
\end{keyword}

\end{frontmatter}

\section{Introduction}
In \cite{hardouin2008differential}, a $\Sigma\Delta\Pi$-Galois theory was developed for $\Sigma\Delta\Pi$-linear systems and was applied to giving a group-theoretic proof of H\"{o}lder's Theorem that the Gamma function is hypertrancendental. Here $\Sigma$ is a set of automorphisms, $\Delta$ is a set of derivations and $\Pi$ is a set of other derivations that endow a differential structure on parameters. Inspired by the work in \cite{hardouin2008differential}, various Galois theories were developed for other kind of systems, for instance, $\sigma$-Galois theory for linear difference equations in \cite{OvchiinnikovWibmer:sigmagaloistheory}, difference Galois theory for linear differential equations in \cite{DivizioHardouinWibmer:differencegaloistheory} and so on. These theories provide powerful tools to study differential/difference algebraic properties of solutions defined by the corresponding linear systems. In particular, using these Galois theories, one is able to prove the hypertranscendence of functions arising in combinatorics, number theory etc,  see \cite{hardouin2008differential,CasaleFreitagNagloo:AxLindemannWeierstrass,DreyfusHardouinRoques:natureofgeneratingseries,HardouinMinchenkoOvchinnikov:differentialgaloisgroups,SchafkeSinger:Consistentsystems,DreysusHardouinRoquesSinger:walksinthequarterplane,AdamczewskiDreyfusHardouin:hypertranscendence,AdamczewskiDreyfusHardouinWibmer:algebraicindependence} and references therein.
 These applications essentially rely on the fact that Galois groups measure differential/difference algebraic relations among solutions, \ie the group is larger, the differential/difference algebraic relations are fewer. On the other hand, the problem of determining Galois groups is one of fundamental problems in Galois theories, which leads to many questions in other areas of mathematics. So far, except for linear differential/difference equations, there is no complete algorithm for computing the Galois groups of other kinds of equations. The readers are referred to \cite{Hendriks:determiningdifferencegaloisgroup,CompointSinger:computinggaloisgroups,DreyfusWeil:computingliealgebra,DreyfusWeil:computingliealgebra,Feng:computationofgaloisgroups,Arreche:computationparameterizedgaloisgroups,Feng:Hrushovskialgorithm,MoulayThomasWeilLucia:computingtheLiealgebra,Dreyfus:computinggaloisgroups,HardouinMinchenkoOvchinnikov:differentialgaloisgroups} and references therein for methods to calculate Galois groups.

 In this paper, we shall focus on $\sigma\delta$-linear systems, \ie systems of linear difference-differential equations with one single automorphism $\sigma$, one single derivative $\delta$ and without parameters. 
Many higher transcendental functions such as Hermite polynomial, Bessel polynomials, Tchebychev polynomials etc satisfy $\sigma\delta$-linear systems (see Chapters 7 and 10 of \cite{ErdelyiMagnusOberhettinger:highertranscendentalfunctions}). For simplicity, we take the Tchebychev polynomials as an example to state the main results of this paper. Let $C$ be an algebraically closed field of characteristic zero and let $C(m,t)$ be the $\sigma\delta$-field with $\sigma(m)=m+1$ and $\delta=d/dt$.  
Consider the Tchebychev polynomials
$$
    T_m(t)=\frac{m}{2}\sum_{\ell=0}^{[\frac{m}{2}]}\frac{(-1)^\ell(m-\ell-1)!}{\ell!(m-2\ell!)}(2t)^{m-2\ell}.
$$
Denote $Y=(T_m(t),T_{m+1}(t))^t$. Then $Y$ satisfies the following $\sigma\delta$-linear system:
\begin{equation}
\label{eqn:Tchebychev}
  \sigma(Y)=A(m,t)Y,\,\,\delta(Y)=B(m,t)Y
\end{equation}
where 
$$
A(m,t)=\begin{pmatrix}
0&1\\
-1&2t
\end{pmatrix},\,\,
B(m,t)=\begin{pmatrix}
\frac{(m-1)t}{1-t^2}&-\frac{m-1}{1-t^2}\\
\frac{m}{1-t^2}&-\frac{mt}{1-t^2}
\end{pmatrix}.
$$
Identified with an algebraic subgroup of $\GL_n(C)$, the $\sigma\delta$-Galois group $G$ of (\ref{eqn:Tchebychev}) over $C(m,t)$ is
$$
  G=\left\{\begin{pmatrix} \xi & 0 \\ 0 & \xi^{-1} \end{pmatrix} \mid \xi\in C^{\times}\right\}\cup \left\{\begin{pmatrix} 0 & \xi \\ \xi^{-1} & 0\end{pmatrix} \mid \xi\in C^{\times}\right\}
$$
(See Example~\ref{exam:sigmadeltagroups}).
Consider $\sigma(Y)=A(m,t)Y$ as a family of linear difference equations with parameter $t$. 
Setting $t$ to be a $c_1\in C\setminus\{\pm 1\}$ yields a linear difference equation $\sigma(Y)=A(m,c_1)Y$ over $C(m)$ whose $\sigma$-Galois group is
$$
  G_{\sigma,c_1}=\begin{cases}  
      \left\{\begin{pmatrix} \xi & 0 \\ 0 & \xi^{-1} \end{pmatrix} \mid \xi^q=1 \right\}& \mbox{$c_1+\sqrt{c_1^2-1}$ is a $q$-th root of unity}\\[4mm]
       \left\{\begin{pmatrix} \xi & 0 \\ 0 & \xi^{-1} \end{pmatrix} \mid \xi\in C^{\times}\right\}
       & \mbox{$c_1+\sqrt{c_1^2-1}$ is not a root of unity}
      \end{cases}
$$
(see Example~\ref{exam:sigmagroups}).
Similarly, setting $m$ to be a $c_2\in C$ yields a linear differential equation $\delta(Y)=B(c_2,t)Y$ over $C(t)$ whose $\delta$-Galois group is
$$
   G_{\delta,c_2}=\begin{cases}
       G & c_2\notin\Q \\
       \left\{\begin{pmatrix} \xi & 0 \\ 0 & \xi^{-1} \end{pmatrix}, \begin{pmatrix} 0 & \xi \\ \xi^{-1} & 0 \end{pmatrix} \mid \xi^q=1\right\} & c_2=\frac{p}{q}\in \Q
    \end{cases}
$$
(see Example~\ref{exam:deltagroups}).
One sees that both $G_{\sigma,c_1}$ and $G_{\delta,c_2}$ are algebraic subgroups of $G$. Moreover setting $U_1=\{c_1\in C \mid \,\,\mbox{$c_1+\sqrt{c_1^2-1}$ is not a root of unity}\}$ and $U_2=C$, one sees that both $U_1$ and $U_2$ are Zariski dense and $G=G_{\sigma,c_1}G_{\delta,c_2}$ for any $(c_1,c_2)\in U_1\times U_2$.
The goal of this paper is to show that these relations still hold true for general $\sigma\delta$-linear systems (see Theorems~\ref{thm:sigmasubgroups}, ~\ref{thm:deltasubgroups} and~\ref{thm:mainresult}, where we use $\stab(\frakm)$ and $\stab(\frakn)$ instead of $G_{\sigma,c_1}$ and $G_{\delta,c_2}$ respectively). We hope these results will be helpful for computing the $\sigma\delta$-Galois groups of linear difference-differential equations. Notice that the relation between the $\delta$-Galois group of $\delta(Y)=B(m,t)Y$ over $\overline{C(m)}(t)$ and the $\delta$-Galois group of $\delta(Y)=B(c_2,t)Y$ over $C(t)$ has been well investigated in \cite{Hrushovski:computingdifferentialgaloisgroups,Feng-Wibmer:differentialgaloisgropus}, and its difference analogue is presnted in \cite{feng2017difference}.

It is well-known that a finite number of elements in a $\delta$-field (resp., $\sigma$-field) are linearly dependent over the field of constants if and only if their Wronskian (resp., Casoration) equals zero (see page 9 of \cite{van2012galois} and page 271 of \cite{Cohn:differencealgebra}). For a $\Delta$-field, Kolchin (see page 86 of \cite{Kolchin:differentialalgebraandalgebraicgroups}) proved the following: a finite number of elements in a $\Delta$-field are linearly dependent over the field of constants if and only if all Wronskian-like determinants vanish. For elements in a $\Sigma\Delta$-ring, the above criteria are not valid in general (see Example~\ref{exam:counterexample}). However, under certain assumptions one can still have similar criteria. For instance, in \cite{LiWuZheng:testinglineardependence}, a criterion for hyperexponential elements in a $\Sigma\Delta$-ring is presented. In this paper, we shall present a criterion for  elements in a simple $\Sigma\Delta$-ring, which generalizes Kolchin's result.

The paper is organized as follows. In Section~\ref{sec:basicdefinitions}, we introduce some basic definitions about $\Sigma\Delta$-linear systems. In Section~\ref{sec:lineardependence}, we present a criterion for testing linear dependence of elements in a simple $\Sigma\Delta$-ring. In Section~\ref{sec:specializations}, we present some properties of the specializations of $\Sigma\Delta$-Picard-Vessiot rings. In Section~\ref{sec:pvrings}, we focus on $\sigma\delta$-Picard-Vessiot rings. We provides a sufficient condition for a $\sigma\delta$-Picard-Vessiot ring to be a $\sigma$-Picard-Vessiot ring. The main results of this paper are presented in Section~\ref{sec:mainresults}.

We are grateful to Carlos Arreche and Michael Wibmer for their valuable comments. 
\section{$\Sigma\Delta$-linear systems}
\label{sec:basicdefinitions}
In this section, we shall recall some basic concepts of $\Sigma\Delta$-linear systems. The readers are referred to the references \cite{hardouin2008differential} for details. All fields in this paper are of characteristic zero.

A $\Sigma\Delta$-ring is a ring $R$ with a set of automorphisms $\Sigma$ and a set of derivations $\Delta$ such that for any $\mu,\tau\in \Sigma\cup \Delta$, $\tau(\mu(r))=\mu(\tau(r))$ for all $r\in R$. The notations of $\Sigma\Delta$-field, $\Sigma\Delta$-ideal, $\Sigma\Delta$-homomorphism, etc. are defined similarly.
The $\Sigma\Delta$-constants $R^{\Sigma\Delta}$ of a $\Sigma\Delta$-ring $R$ is the set
$$
  R^{\Sigma\Delta}=\{r\in R \mid \sigma(r)=r \,\,\forall \sigma\in \Sigma,\mbox{and}\,\, \delta(r)=0 \,\,\forall \delta\in \Delta \}.
$$
A simple $\Sigma\Delta$-ring is a $\Sigma\Delta$-ring whose only $\Sigma\Delta$-ideals are $(0)$ and $R$.
Given a $\Sigma\Delta$-field $k$, a $\Sigma\Delta$-linear system is a system of equations of the form
\begin{equation}
\label{eqn:SigmaDeltasystem}
\begin{cases}
    \sigma_i(Y)=A_i Y,& A_i\in \GL_n(k),\,\, \forall\,\sigma_i\in \Sigma\\
    \delta_i(Y)=B_i Y,& B_i \in \gl_n(k),\,\,\forall\,\delta_i\in \Delta
\end{cases}
\end{equation}
where the $A_i, B_j$ satisfy the integrability condition:
\begin{align*}
  \sigma_i(A_j)A_i&=\sigma_j(A_i)A_j,\\
 \sigma_i(B_j)A_i&=\delta_j(A_i)+A_i B_j,\\
 \delta_i(B_j)+B_jB_i&=\delta_j(B_i)+B_iB_j
\end{align*}
for all $\sigma_i,\sigma_j\in \Sigma$ and all $\delta_i,\delta_j\in \Delta$. Assume that $k$ is a $\Sigma\Delta$-field.
\begin{defi}[Definition 6.10 of \cite{hardouin2008differential} with $\Pi=\emptyset$]
\label{def:sigmadeltaPVextension}
A $\Sigma\Delta$-ring $R$ is a $\Sigma\Delta$-Picard-Vessiot ring over a $\Sigma\Delta$-field $k$ for (\ref{eqn:SigmaDeltasystem}) if
\begin{enumerate}
\item  $R$ is a simple $\Sigma\Delta$-ring, and
 \item $R=k[\calX,\frac{1}{\det(\calX)}]$ where $\calX\in \GL_n(R)$ satisfies $\sigma_i(\calX)=A_i \calX \,\,\forall \sigma_i\in \Sigma$ and $\delta_i(\calX)=B_i \calX\,\,\forall \delta_i\in \Delta$.
\end{enumerate}
\end{defi}
  The invertible matrix $\calX$ in Definition~\ref{def:sigmadeltaPVextension} is usually called a fundamental solution matrix of the corresponding $\Sigma\Delta$-linear system.
\begin{defi}
Suppose that $R$ is a $\Sigma\Delta$-Picard-Vessiot ring over $k$ for (\ref{eqn:SigmaDeltasystem}).
The $\Sigma\Delta$-Galois group of (\ref{eqn:SigmaDeltasystem}) over $k$ (or $R$ over $k$) is defined to be the set of $\Sigma\Delta$-$k$-automorphisms of $R$ over $k$, denoted by $\Sigma\Delta$-$\Gal(R/k)$.
\end{defi}

 A $\Sigma\Delta$-Picard-Vessiot ring $R$ over $k$ for (\ref{eqn:SigmaDeltasystem}) alway exists. When $k^{\Sigma\Delta}$ is algebraically closed, $R$ is unique up to $\Sigma\Delta$-$k$-isomorphisms (see Proposition 6.16 of \cite{hardouin2008differential} for a proof), and $\SDgal(R/k)$ can be identified with an algebraic subgroup of $\GL_n(k^{\Sigma\Delta})$ defined over $k^{\Sigma\Delta}$ (see Proposition 6.18 of \cite{hardouin2008differential} for a proof). The second assertion is still true if the condition that $k^{\Sigma\Delta}$ is algebraically closed is replaced with $R^{\Sigma\Delta}=k^{\Sigma\Delta}$. For the difference case, \ie $\Sigma=\{\sigma\}$ and $\Delta=\emptyset$, this has already been proved in \cite{chatzidakis2007definitions,wibmerthesis}. 
 We shall prove the general case. Note that by Proposition 6.14 of \cite{hardouin2008differential} with $\Pi=\emptyset$, if $k^{\Sigma\Delta}$ is algebraically closed then $R^{\Sigma\Delta}=k^{\Sigma\Delta}$.

 Note that for every $\Sigma\Delta$-simple ring $R$, $R^{\Sigma\Delta}$ is a field (see Lemma 1.7 on page 6 of \cite{van2006galois} for a proof of difference case). 
\begin{lem}
\label{lm:linearlydisjoint}
Suppose that $S\subseteq T$ are two $\Sigma\Delta$-rings and $S$ is $\Sigma\Delta$-simple. Then $S$ and $T^{\Sigma\Delta}$ are linearly disjoint over $S^{\Sigma\Delta}$.
\end{lem}
\begin{proof}
Use an argument similar to the proof of Lemma 1.1.6 of \cite{wibmerthesis}.
\end{proof}
Given two $n\times n$ matrices $(a_{i,j}), (b_{i,j})$, we shall use $(a_{i,j})\otimes (b_{i,j})$ to stand for the matrix $(\sum_{l=1}^n a_{i,l}\otimes b_{l,j})$.
 \begin{prop}
 \label{prop:galoisgroups} Suppose that $R$ is a $\Sigma\Delta$-Picard-Vessiot ring over $k$ for (\ref{eqn:SigmaDeltasystem}) and $R^{\Sigma\Delta}=k^{\Sigma\Delta}$. Then $\SDgal(R/k)$ can be identified with the set of $k^{\Sigma\Delta}$-points of an affine algebraic group defined over $k^{\Sigma\Delta}$ with $(R\otimes_k R)^{\Sigma\Delta}$ as its coordinate ring.
 \end{prop}
\begin{proof}
Let $\calX$ be a fundamental solution matrix in $\GL_n(R)$. Set $\calZ=\calX^{-1}\otimes \calX$. Then $\calZ\in \GL_n((R\otimes_k R)^{\Sigma\Delta})$ and moreover one can verity that $(R\otimes_k R)^{\Sigma\Delta}=k^{\Sigma\Delta}[\calZ,1/\det(\calZ)]$. Denote $S=k^{\Sigma\Delta}[\calZ,1/\det(\calZ)]$. Note that $R$ can be viewed as a subring of $R\otimes_k R$. Due to Lemma~\ref{lm:linearlydisjoint}, the homomorphism $R\otimes_{k^{\Sigma\Delta}} S\rightarrow R\otimes_k R$ given by $a\otimes b\mapsto (a\otimes 1)b$ is $\Sigma\Delta$-isomorphic. Denote by $\varphi$ the inverse map of the above isomorphism, \ie 
 \begin{equation}
 \label{eqn:varphi}
\varphi: \xymatrix@R=3pt{ R\otimes_k R \ar[r] &  R\otimes_{k^{\Sigma\Delta}} S \\
                             a\otimes b \ar@{|->}[r] &  (a\otimes 1)b(\calX\otimes \calZ)}.
\end{equation}
and denote by $\rho$ the composite of $\varphi$ with the embedding $R\rightarrow R\otimes_k R, a\mapsto 1\otimes a$. Then we have the following $\Sigma\Delta$-homomorphism:
\[
\id\otimes\rho: \xymatrix@R=3pt{ R\otimes_k R \ar[r] & R\otimes_k R\otimes_{k^{\Sigma\Delta}} S \\
                             a\otimes b \ar@{|->}[r] &  a\otimes b(\calX\otimes \calZ)}.
\]
Note that $(R\otimes_k R\otimes_{k^{\Sigma\Delta}} S)^{\Sigma\Delta}=S\otimes_{k^{\Sigma\Delta}} S$. Thus $\id\otimes \rho$ induces the comultiplication map $\Delta: S\rightarrow S\otimes_{k^{\Sigma\Delta}} S, \calZ\mapsto \calZ\otimes \calZ$. The $\Sigma\Delta$-homomorphism $R\otimes_k R\rightarrow R, a\otimes b\mapsto ab$ induces the counit $\epsilon: S\rightarrow k^{\Sigma\Delta}$. These make $S$ become the coordinate ring of an affine algebraic group over $k^{\Sigma\Delta}$.

It remains to show that there is a group isomorphism from $\SDgal(R/k)$ to $\Hom_{k^{\Sigma\Delta}}(S, k^{\Sigma\Delta})$. For $\tau\in \SDgal(R/k)$, we define
\[
\phi_\tau: \xymatrix@R=3pt{S \ar@{^{(}->}[r] & R\otimes_k R \ar[r]^{\id\otimes \tau} & R\otimes_k R \ar[r] & R \\
                             \calZ \ar@{^{(}->}[r] &   \calX^{-1} \otimes \calX  \ar@{|->}[r] & \calX^{-1} \otimes \tau(\calX) \ar@{|->}[r] & \calX^{-1}\tau(\calX) }.
\]
Then $\phi_\tau\in \Hom_{k^{\Sigma\Delta}}(S,k^{\Sigma\Delta})$ and the map $\phi: \SDgal(R/k)\rightarrow \Hom_{k^{\Sigma\Delta}}(S,k^{\Sigma\Delta})$ given by $\tau\mapsto \phi_\tau$ is a group homomorphism. If $\phi_\tau=\epsilon$ then $\tau(\calX)=\calX$, \ie $\tau=\id$ and so $\phi$ is injective. Now suppose $\gamma\in \Hom_{k^{\Sigma\Delta}}(S,k^{\Sigma\Delta})$. We have the following $\Sigma\Delta$-homomorphism:
\begin{equation*}
\label{eqn:rho}
\rho_\gamma: \xymatrix@R=3pt{R \ar[r] & R\otimes_k R \ar[r] & R\otimes_{k^{\Sigma\Delta}} S \ar[r] & R \\
                             \calX \ar@{|->}[r] &   1 \otimes \calX  \ar@{|->}[r] & \calX \otimes \calZ \ar@{|->}[r] & \calX \gamma(\calZ)}.
\end{equation*}
Since $\gamma(\calZ)$ is invertible, $\rho_\gamma$ is surjective. Furthermore, as $R$ is $\Sigma\Delta$-simple, $\rho_\gamma$ is injective. Therefore $\rho_\gamma\in \SDgal(R/k)$. One can verify that $\phi(\rho_\gamma)=\gamma$. Therefore $\phi$ is surjective and so $\phi$ is isomorphic.
\end{proof}

\begin{rem}
The proof of the above proposition implies that $R$ is actually the coordinate ring of some $\SDgal(R/k)$-torsor over $k$.
\end{rem}

For the remainder of this paper, when we speak of the $\Sigma\Delta$-Galois group of $R$ over $k$, we usually refer to $\Hom_{k^{\Sigma\Delta}}((R\otimes_k R)^{\Sigma\Delta}, k^{\Sigma\Delta})$. Let $\K$ be the total ring of fractions of $R$. Then each $\tau\in \SDgal(R/k)$ can be uniquely extended into an antomorphism of $\K$ over $k$. We still use $\tau$ to denote this extended automorphism and we have that $
\SDgal(R/k)=\SDgal(\K/k)$. 

\section{Linear dependence of elements in a simple $\Sigma\Delta$-ring}
\label{sec:lineardependence}
In this section, we shall give a criterion for testing linear dependence of elements in a simple $\Sigma\Delta$-ring. This criterion will be used in the later sections and it may be of independent interest. The following notation will be used frequently.
\begin{notation}
$\Theta$ is the semigroup generated by $\Sigma\cup \Delta$.
\end{notation}
Suppose that $R$ is a simple $\Sigma\Delta$-ring and $a_1,\dots,a_m\in R$. Let $W$ be a subset of $R[X_1,\dots,X_m]$, where $X_1,\dots,X_m$ are indeterminates. We say $W$ is closed under the action of $\Theta$ if $f_\theta\in W$ for all $\theta\in \Theta, f\in W$, where $f_\theta$ stands for the polynomial obtained by applying $\theta$ to the coefficients of $f$.
\begin{lem}
\label{lm:solutions}
Suppose that $W\subset R[X_1,\dots,X_m]$ is a set of linear homogeneous polynomials and $W$ is closed under the action of $\Theta$. Then $W$ has a nontrivial zero in $R^m$ if and only if for any $f_1,\dots,f_m\in W$, $\det(M_{f_1,\dots,f_m})=0$, where $M_{f_1,\dots,f_m}$ stands for the coefficient matrix of $f_1,\dots,f_m$.
\end{lem}
\begin{proof}
We prove the sufficiency by induction on $m$. When $m=1$, the assertion is obviously true since $W=\{0\}$. Suppose that $m>1$. Let $\tilde{W}$ be the $R$-module generated by $W$. Then it suffices to show that $\tilde{W}$ has a nontrivial zero in $R^m$. Since $W$ is closed under the action of $\Theta$, so is $\tilde{W}$. Set
$$
   \tilde{W}_m=\tilde{W}\cap R[X_1,\dots,X_{m-1}].
$$
Then $\tilde{W}_m$ is also a $R$-module that is closed under the action of $\Theta$.
If $\tilde{W}_m=\tilde{W}$ then $(0,\dots,0,1)$ is a nontrivial zero of $\tilde{W}$ and the assertion holds. Suppose that $\tilde{W}_m\neq\tilde{W}$. We claim that for any $h_1,\dots,h_{m-1}\in \tilde{W}_m$, $\det(M_{h_1,\dots,h_{m-1}})=0$. Let $f=a_1X_1+\dots+a_mX_m\in \tilde{W}\setminus\tilde{W}_m$. Then $a_m\neq 0$. Since $R$ is $\Sigma\Delta$-simple, there are $b_1,\dots,b_\ell\in R$ and $\theta_1,\dots,\theta_\ell\in \Theta$ such that $\sum_{i=1}^\ell b_i\theta_i(a_m)=1$. Then $h_m=\sum_{i=1}^\ell b_i f_{\theta_i}=\tilde{a}_1X_1+\tilde{a}_2X_2+\dots+X_m\in \tilde{W}$. Suppose that $h_1,\dots,h_{m-1}\in \tilde{W}_m$. It is easy to see that $\det(M_{h_1,\dots,h_{m-1}})=\det(M_{h_1,\dots,h_m})$. Write $h_i=\sum_{j=1}^s c_{ij} f_j$ with $c_{ij}\in R$, where $i=1,2,\dots,m$ and $f_j\in W$. Without loss of generality, we may assume that $s\geq m$. Let $D=(c_{ij})_{1\leq i \leq m,1\leq j\leq s}$ and $T=M_{f_1,\dots,f_s}$. Then $M_{h_1,\dots,h_m}=DT$ and by Cauchy-Binet formula,
\begin{align*}
   \det(M_{h_1,\dots,h_m})&=\det(DT)\\
   &=\sum_{1\leq l_1<\dots<l_m\leq s} D\left(\begin{array}{cccc} l_1  & \dots&l_m \\ 1& \dots & m\end{array}\right) T \left(\begin{array}{cccc} 1  & \dots&m \\ l_1& \dots & l_m\end{array}\right)=0
\end{align*}
where $D(\cdot), T(\cdot)$ denotes the $m$ order minors of $D$ and $T$ respectively. The last equality holds because $T \left(\begin{array}{cccc} 1  & \dots&m \\ l_1& \dots & l_m\end{array}\right)=\det(M_{f_{l_1},\dots,f_{l_m}})=0$. Hence $\det(M_{h_1,\dots,h_{m-1}})=0$. This proves our claim. By induction hypothesis, $\tilde{W}_m$ has a nontrivial zero in $R^{m-1}$, say $(c_1,\dots,c_{m-1})$. Set $c_m=-\sum_{i=1}^{m-1}\tilde{a}_ic_i$. Then $h_m(c_1,\dots,c_m)=0$. For each $g\in \tilde{W}$, one has that $g-b_g h_m\in \tilde{W}_m$ where $b_g$ is the coefficient of $X_m$ in $g$. This implies that $g(c_1,\dots,c_m)-b_gh_m(c_1,\dots,c_m)=0$. Consequently, $g(c_1,\dots,c_m)=0$ and thus $(c_1,\dots,c_m)$ is a nontrivial zero of $W$.

Assume that $W$ has a nontrivial zero $(a_1,\dots,a_m)$ in $R^m$. Without loss of generality, we may assume that $a_1\neq 0$. For any $g_1,\dots,g_m\in W$, one has that
$$
   M_{g_1,\dots,g_m}^* (g_1,\dots,g_m)^t=\det(M_{g_1,\dots,g_m})(X_1,\dots,X_m)^t
$$
where $(\cdot)^*$ stands for the adjoint matrix. Substituting $a_i$ for $X_i$ in the above equality yields
$
\det(M_{g_1,\dots,g_m})a_i=0, \forall\,\,g_1,\dots,g_m\in W, \forall\,\,i=1,\dots,m.
$
In particular
\begin{equation}
\label{eqn:determinant}
   \det(M_{g_1,\dots,g_m})a_1=0, \forall\,\,g_1,\dots,g_m\in W.
\end{equation}
Suppose that there exist $f_1,\dots,f_m\in W$ such that $d=\det(M(f_1,\dots,f_m))\neq 0$. For each $\sigma\in \Sigma$, since $f_{i,\sigma}\in W$ and $M_{f_{1,\sigma},\dots,f_{m,\sigma}}=\sigma(M_{f_1,\dots,f_m})$, by (\ref{eqn:determinant}), $\sigma(d)a_1=0$. For each $\delta\in \Delta$, since
$$
   \delta(d)=\sum_{i=1}^m \det(M_{f_1,\dots,f_{i,\delta},\dots,f_m}),
$$
one has that $\delta(d)a_1=0$. Repeating the above process yields that $\theta(d)a_1=0$ for all $\theta\in \Theta$. Hence the ideal $\{b\in R \mid ba_1=0\}\subseteq R$ contains the $\Sigma\Delta$-ideal generated by $d$. Since $R$ is $\Sigma\Delta$-simple and $d\neq 0$, $1\in \{b\in R \mid ba_1=0\}$. This implies that $a_1=0$, a contradiction. Therefore for any $f_1,\dots,f_m\in W$, $\det(M_{f_1,\dots,f_m})=0$.
\end{proof}

\begin{prop}
\label{prop:lineardependence}
Suppose that $R$ is a simple $\Sigma\Delta$-ring and $a_1,\dots,a_m\in R$. Then
 $a_1,\dots,a_m$ are linearly dependent over $R^{\Sigma\Delta}$ if and only if for any $\theta_1$, $\dots,\theta_m\in \Theta$, we have
$\det((\theta_i(a_j))_{1\leq i,j \leq m})=0.$
\end{prop}
\begin{proof}
Suppose that $c_1a_1+\dots+c_ma_m=0$ for some $c_1,\dots,c_m\in R^{\Sigma\Delta}$, not all zero. Then for any $\theta_1,\dots,\theta_m\in \Theta$, $c_1\theta_i(a_1)+\dots+c_m\theta_i(a_m)=0$ for all $i=1,\dots,m$. In matrix form, $(\theta_i(a_j))\bfc=0$ where $\bfc=(c_1,\dots,c_m)^t$. Multiplying the adjoint matrix of $(\theta_i(a_j))$, we obtain $\det((\theta_i(a_i)))\bfc=0$. Since not all $c_i$ are zero, $\det((\theta_i(a_j)))=0$.

Conversely, suppose that for any $\theta_1,\dots,\theta_m\in \Theta$, $\det((\theta_i(a_i)))=0$.
We shall show that $a_1,\dots,a_m$ are linearly dependent over $R^{\Sigma\Delta}$ by induction on $m$. If $m=1$ there is nothing to prove. Assume that $m>1$. Set $W=\{\sum_{i=1}^m \theta(a_i)X_i \mid \forall\,\,\theta\in \Theta\}$. Then $W$ is closed under the action of $\Theta$. By Lemma~\ref{lm:solutions}, $W$ has a nontrivial zero in $R^m$.
Let $U$ be the set of all zeroes of $W$ in $R^m$. Then $U\neq \{(0,\dots,0)\}$. Let $\Theta^{-}$ be the semigroup generated by $\{\sigma^{-1} \mid \forall\,\,\sigma\in \Sigma\}\cup \Delta$. We first show that for any $\theta\in \Theta^{-}$, $\theta(U)\subset U$. To this end, it suffices to show that for any $\sigma\in \Sigma, \delta\in \Delta$, $\delta(U),\sigma^{-1}(U)\subset U$ and $\delta(U)\subset U$. Suppose that $(b_1,\dots,b_m)\in U\setminus\{(0,\dots,0)\}$. Then $\sum_{i=1}^m\theta(a_i)b_i=0$ for all $\theta\in \Theta$. Applying $\delta\in \Delta$ to both sides yields that
$$
0=\sum_{i=1}^m \delta(\theta(a_i)b_i)=\sum_{i=1}^m (\delta\theta(a_i)b_i+\theta(a_i)\delta(b_i)).
$$
 This implies that $\sum_{i=1}^m \theta(a_i)\delta(b_i)=0$ because $\sum_{i=1}^m \delta\theta(a_i)b_i=0$. In other words, $(\delta(b_1),\dots,\delta(b_m))\in U$ and so $\delta(U)\subset U$.
 Since $\sum_{i=1}^m \sigma\theta(a_i)b_i=0$ for any $\sigma\in \Sigma$, applying $\sigma^{-1}$ to both sides yields that $\sum_{i=1}^m \theta(a_i)\sigma^{-1}(b_i)=0$. So $\sigma^{-1}((b_1,\dots,b_m))\in U$ and then $\sigma^{-1}(U)\subset U$. Consequently, $\theta(U)\subseteq U$ for any $\theta\in \Theta^{-}$. Next, we shall show that $U$ contains an element with 1 as some coordinate. Suppose $(b_1,\dots,b_m)\in U\setminus\{(0,\dots,0)\}$. Without loss of generality, we may assume that $b_m\neq 0$. Since $R$ is $\Sigma\Delta$-simple, there are $\theta_1,\dots,\theta_s\in \Theta$ and $u_1,\dots,u_s\in R$ such that $\sum_{i=1}^s u_i\theta_i(b_m)=1$. Let $\tau$ be an element in the semigroup generated by $\{\sigma^{-1} \mid \forall\,\,\sigma\in \Sigma\}$ such that $\tau\theta_i\in \Theta^{-}$ for all $i=1,\dots,s$. Set $\tilde{b}_j=\sum_{i=1}^s \tau(u_i)\tau\theta_i(b_j)$. Then $\tilde{b}_m=1$ and $\tilde{\bfb}=(\tilde{b}_1,\dots,\tilde{b}_m)\in U$ because $U$ is a $R$-module and $\theta(U)\subseteq U$ for any $\theta\in\Theta^{-}$. Thus $\tilde{\bfb}$ satisfies the rquired property. Finally, consider the set
 $$
    S=\{\sigma^{-1}(\tilde{\bfb})-\tilde{\bfb} \mid \forall\,\,\sigma\in \Sigma\}\cup \{\delta(\tilde{\bfb}) \mid \forall\,\,\delta\in \Delta\}.
 $$
If $S=\{(0,\dots,0)\}$ then $\tilde{\bfb}\in (R^{\Sigma\Delta})^m$. So $a_1,\dots,a_m$ are linearly dependent over $R^{\Sigma\Delta}$ as $\sum_{i=1}^m a_i\tilde{b}_i=0$. Assume that $S\neq \{(0,\dots,0)\}$ and $(c_1,\dots,c_m)\in S\setminus\{(0,\dots,0)\}$. Since $S\subset U$, $(c_1,\dots,c_m)$ is a nontrivial zero of $W$. Furthermore, as $\tilde{b}_m=1$, one sees that $c_m=0$. Therefore $(c_1,\dots,c_{m-1})$ is a nontrivial zero of $W_{m-1}=\{\sum_{i=1}^{m-1} \theta(a_i)X_i \mid \forall\,\, \theta\in \Theta\}$. Since $W_{m-1}$ is closed under the action of $\Theta$, by Lemma~\ref{lm:solutions}, $\det((\theta_i(a_j))_{1\leq i,j\leq m-1})=0$ for any $\theta_1,\dots,\theta_{m-1}\in \Theta$. Using induction hypothesis, $a_1,\dots,a_{m-1}$ are linearly dependent over $R^{\Sigma\Delta}$ and so are $a_1,\dots,a_m$.
\end{proof}

The example below shows that the above proposition is not true if $R$ is not $\Sigma\Delta$-simple.
\begin{exam}
\label{exam:counterexample}
Let $R=\Q[y,z]$ where $y,z$ are two indeterminates. Define $\sigma: R\rightarrow R$ as follows: $\sigma(c)=c\,\,\forall\,c\in \Q, \sigma(y)=2y, \sigma(z)=2z$. Then $R$ is a $\sigma$-ring but not $\sigma$-simple. One can verify that $R^\sigma=\Q$ and for any $i,j\geq 0$, $\sigma^i(y)\sigma^j(z)-\sigma^i(z)\sigma^j(y)=0$. However, $y,z$ are linearly indepdent over $\Q$.
\end{exam}
\begin{cor}
\label{cor:linearindependence}
Suppose that $\Sigma\subseteq \{\sigma\}$, $R$ is a simple $\Sigma\Delta$-ring and $a_1,\dots,a_m\in R$. If $a_1,\dots, a_m$ are linearly independent over $R^{\Sigma\Delta}$ then there are $\theta_1,\dots, \theta_m$ with $\theta_1=1$ such that $\det((\theta_i(a_j))_{1\leq i,j\leq m})\neq 0$.
\end{cor}
\begin{proof}
For any $\theta_1,\dots,\theta_m\in \Theta$, set
$$
d(\theta_1,\dots,\theta_m)=\det((\theta_i(a_j))_{1\leq i,j\leq m}).
$$
Suppose that $d(1,\theta_2,\dots,\theta_m)=0$ for any $\theta_2,\dots, \theta_m\in \Theta$. We need to show that $a_1,\dots,a_m$ are linearly dependent over $R^{\Sigma\Delta}$. Due to Proposition~\ref{prop:lineardependence}, we only need to show that $d(\theta_1,\dots,\theta_m)=0$ for any $\theta_1,\dots, \theta_m\in \Theta$. To this end, note that for any $\delta\in \Delta$,
$$
   \delta(d(\theta_1,\theta_2,\dots,\theta_m))=\sum_{i=1}^m d(\theta_1, \theta_2,\dots,\delta\theta_i,\dots,\theta_m).
$$
The above equality with $\theta_1=1$ implies that $d(\delta,\theta_2,\dots,\theta_m)=0$ for any $\theta_2,\dots,\theta_m\in \Theta$ and any $\delta\in \Delta$. Using the above equality repeatedly, we have that $d(\theta_1,\theta_2,\dots,\theta_m)=0$ for any $\theta_2,\dots,\theta_m\in \Theta$ and any $\theta_1 \in \Theta_\Delta$, the semigroup generated by $\Delta$. If $\Sigma=\emptyset$ then we are done. Suppose $\Sigma=\{\sigma\}$ and $\theta_1,\dots,\theta_m\in \Theta$. Then there is $j\geq 0$ such that $\theta_i=\sigma^j\bar{\theta_i}$ where $\bar{\theta_i}\in \Theta$ and at least one of them is in $\Theta_\Delta$, say $\bar{\theta}_\ell$. The previous discussion implies that $d(\bar{\theta}_1,\dots,\bar{\theta}_m)=0$ as $\bar{\theta}_\ell\in \Theta_\Delta$. So $d(\theta_1,\dots,\theta_m)=\sigma^j(d(\bar{\theta}_1,\dots,\bar{\theta}_m))=0$.
\end{proof}

\begin{prop}
\label{prop:linearindependence2}
Suppose that $\Sigma\subseteq \{\sigma\}, \Delta\subseteq\{\delta\}$ and $k$ is a $\Sigma\Delta$-field with algebraically closed $k^{\Sigma\Delta}$. Assume that $R$ is a $\Sigma\Delta$-Picard-Vessiot ring over $k$ for some $\Sigma\Delta$-linear system. Let $V\subset R$ be a $\SDgal(R/k)$-invariant $k^{\Sigma\Delta}$-vector space of finite dimension and
$\{a_1,\dots,a_m\}$ a basis of $V$. Then there are $\theta_1,\dots,\theta_m\in \Theta$ with $\theta_1=1$ such that $\det((\theta_i(a_j))_{1\leq i,j\leq m})$ is invertible in $R$.
\end{prop}
\begin{proof}
 By Corollary~\ref{cor:linearindependence}, there are $\theta_1,\dots,\theta_m\in \Theta$ with $\theta_1=1$ such that $d=\det((\theta_i(a_j))_{1\leq i,j\leq m})\neq 0$. For each $g\in \SDgal(R/k)$, $g(d)=c_g d$ where $c_g\in k^{\Sigma\Delta}$. By Lemma 6.8 of \cite{hardouin2008differential} and its proof with $\Sigma\subseteq \{\sigma\}, \Delta\subseteq\{\delta\}$ and $\Pi=\emptyset$, there exist idempotents $e_0,\dots,e_{s-1}$ such that
\begin{enumerate}
   \item $R=Re_0\oplus \dots \oplus Re_{s-1}$,
   \item each $e_iR$ is an integral domain and $\sigma^s\delta$-simple,
   \item there is $h\in \SDgal(R/k)$ such that $h$ permutes $e_0,\dots,e_{s-1}$.
\end{enumerate}
Due to Lemma 6 of \cite{feng2010liouvillian}, $\Sigma^s\Delta$-$\Gal(e_0R/k)$ is an algebraic subgroup of $\SDgal(R/k)$. Since $d\neq 0$, $de_0\neq 0$. Otherwise, applying $h$ to $de_0$ yields that $de_i=0$ for all $i$ and thus $d=0$, a contradiction. Now for each $g\in \Sigma^s\Delta$-$\Gal(e_0R/k)$, $g(de_0)=c_g de_0$. As $Re_0$ is an integral domain, both $\sigma^s(de_0)/de_0$ and $\delta(de_0)/de_0$ are in the field $\K_0$ of quotients of $Re_0$. Since for each $g\in \Sigma^s\Delta$-$\Gal(e_0R/k)=\Sigma^s\Delta$-$\Gal(\K_0/k)$,
$$
  g(\tilde{\sigma}(de_0)/de_0)=\sigma^s(de_0)/de_0,\quad g(\delta(de_0)/de_0)=\delta(de_0)/de_0
$$
for any $\tilde{\sigma}\in \Sigma^s$ and any $\delta\in \Delta$. The Galois correspondence implies that both $\tilde{\sigma}(de_0)/de_0$ and $\delta(de_0)/de_0$ are in $k$ for any $\tilde{\sigma}\in \Sigma^s$ and any $\delta\in \Delta$. This implies the ideal $(de_0)$ of $Re_0$ generated by $de_0$ is a nontrivial $\Sigma^s\Delta$-ideal. Since $Re_0$ is $\Sigma^s\Delta$-simple, $de_0$ is invertible in $Re_0$, i.e. there is $u_0e_0\in Re_0$ such that $u_0de_0=e_0$. Applying $h$ to $u_0de_0=e_0$ repeatedly implies that there are $u_ie_i\in Re_i, i=0,\dots,s-1$ such that $u_ide_i=e_i$ for all $i$. Set $u=\sum_{i=0}^{s-1}u_ie_i\in R$. Then
$$
   ud=\left(\sum_{i=0}^{s-1}u_ie_i\right)\left(\sum_{i=0}^{s-1} de_i\right)=\sum_{i=0}^{s-1}u_ide_i=\sum_{i=0}^{s-1}e_i=1,
$$
i.e. $d$ is invertible in $R$.
\end{proof}

\begin{cor}
\label{cor:groupinvariantideal}
Suppose that $\Sigma\subseteq \{\sigma\}, \Delta\subseteq \{\delta\}$ and $k$ is a $\Sigma\Delta$-field with algebraically closed $k^{\Sigma\Delta}$. Assume that $R$ is a $\Sigma\Delta$-Picard-Vessiot ring over $k$ for some $\Sigma\Delta$-linear system. Then $R$ has no nontrivial $\SDgal(R/k)$-invariant ideal.
\end{cor}
\begin{proof}
Suppose that $I$ is a $\SDgal(R/k)$-invariant ideal and $I\neq (0)$. We shall show that $I=R$. Let $a\in I\setminus\{0\}$ and let $\{a_1,\dots,a_m\}$ be a basis of the $k^{\Sigma\Delta}$-vector space spanned by $\{g(a) \mid g\in \SDgal(R/k)\}$. By Proposition~\ref{prop:linearindependence2}, there are $\theta_1,\dots,\theta_m\in \Theta$ with $\theta_1=1$ such that $d=\det((\theta_i(a_j))_{1\leq i,j\leq m})$ is invertible in $R$. Expanding $d$ by the first row, one sees that $d\in I$. This implies that $I=R$.
\end{proof}

\section{Specializations of $\Sigma\Delta$-Picard-Vessoit rings}
\label{sec:specializations}
Thoughout this section, $k$ is a $\Sigma\Delta$-field with algebraically closed field of constants $C=k^{\Sigma\Delta}$, $R$ is a $\Sigma\Delta$-Picard-Vessiot ring for (\ref{eqn:SigmaDeltasystem}) over $k$ and $G=\Hom_C((R\otimes_k R)^{\Sigma\Delta},C)$. As shown in Proposition~\ref{prop:galoisgroups}, we may identify the $\Sigma\Delta$-Galois group $\SDgal(R/k)$ with $G$. We shall fix a fundamental solution matrix $\calX$ of (\ref{eqn:SigmaDeltasystem}) in $\GL_n(R)$, and set $\calZ=\calX^{-1}\otimes_k \calX$. Then $C[G]=(R\otimes_k R)^{\Sigma\Delta}=C[\calZ,1/\det(\calZ)]$.

We shall investigate the specializations of $R$. These specializations play an important role to connect $G$ to the Galois groups of the specializations of the linear difference equation and the differential equation in (\ref{eqn:sigmadelta-eqn}) respectively. To construct the specializations of $R$, we need to introduce a simple-$\Sigma\Delta$ subring of $R$. 
We assume that
\begin{itemize}
  \item $D$ is simple $\Sigma\Delta$-ring such that $k$ is the field of fractions of $D$,
  \item $\R=D[\calX,1/\det(\calX)]$.
  \end{itemize}
Note that $D$ in the above assumptions always exists, for instance, we may simply set $D=k$. Then $D^{\Sigma\Delta}=k^{\Sigma\Delta}=C$ due to \cite{wibmerthesis}. We shall use $\R$ to construct the Picard-Vessiot rings corresponding to the specializations of the equations in (\ref{eqn:sigmadelta-eqn}). Let us start with a lemma that has already appeared in the literature (see for example \cite{hardouin2008differential}, Lemma 1.23 of \cite{van2012galois}, Lemma 1.11 of \cite{van2006galois} and Proposition 1.4.15 of \cite{wibmerthesis}) for special cases.

\begin{lem}
\label{lm:correspondence}
Suppose that $S\subseteq T$ are two $\Sigma\Delta$-rings and $S$ is $\Sigma\Delta$-simple. Assume further that $T$ is generated by $T^{\Sigma\Delta}$ as an $S$-module. Then the map $J\rightarrow (J)$ is a bijective correspondence from the set of ideals of $T^{\Sigma\Delta}$ to the set of $\Sigma\Delta$-ideals of $T$.
\end{lem}
\begin{proof}
Suppose that $I$ is a $\Sigma\Delta$-ideal of $T$. It suffices to show that $I$ is generated by $I\cap T^{\Sigma\Delta}$. Any $f\in I$ can be written as $f=\sum_{i=1}^s a_ib_i$ where $a_i\in S, b_i\in T^{\Sigma\Delta}$ and $a_1,\dots,a_s$ are linearly independent over $S^{\Sigma\Delta}$. By Proposition~\ref{prop:lineardependence}, there are $\theta_1,\dots,\theta_s\in \Theta$ such that $d=\det((\theta_i(a_j)))\neq 0$. We have that $(\theta_1(f),\dots,\theta_s(f))^t=(\theta_i(a_j))(b_1,\dots,b_s)^t$. Multiplying the adjoint matrix of $(\theta_i(a_j))$ on both sides of the previous linear equations yields that $db_i\in I$ for any $i=1,\dots,s$. Hence $\Gamma_i=\{a\in S\mid ab_i\in I\}$ is a nonzero ideal. As $b_i$ is constant, $\Gamma_i$ is a nonzero $\Sigma\Delta$-ideal. Since $S$ is $\Sigma\Delta$-simple, $1\in \Gamma_i$, \ie $b_i\in I\cap T^{\Sigma\Delta}$. Consequently, $f$ belongs to the ideal generated by $I\cap T^{\Sigma\Delta}$.
\end{proof}

\begin{lem}
\label{lm:ringinjectivity}
The natural homomorphism $\fraki: \R\otimes_D \R\rightarrow R\otimes_k R$ given by $a\otimes_D b\mapsto a\otimes_k b$ is injective.
\end{lem}
\begin{proof}
We first claim that $\R$ is $\Sigma\Delta$-simple.
Suppose that $I$ is a nonzero $\Sigma\Delta$-ideal of $\R$. It suffices to show that $1\in I$. Let $a\in I\setminus\{0\}$. Since $R$ is $\Sigma\Delta$-simple, there are $b_1,\dots,b_s \in R$ and $\theta_1,\dots,\theta_s\in \Theta$ such that $\sum_{i=1}^s b_i\theta_i(a)=1$. Let $p\in D$ be nonzero such that $pb_i\in \R$ for all $i$. Then $p=\sum_{i=1}^s pb_i\theta_i(a)\in I\cap D$. Hence $I\cap D$ is a nonzero $\Sigma\Delta$-ideal of $D$ and so $1\in I\cap D$ because $D$ is $\Sigma\Delta$-simple. Consequently, $1\in I$. This proves our claim.

Since $\R$ is $\Sigma\Delta$-simple, the map $\R\rightarrow \R\otimes_D \R, a\mapsto a\otimes 1$ is injective and thus $\R$ can be viewed as a $\Sigma\Delta$-subring of $\R\otimes_D \R$. We have that $(\R\otimes_D \R)^{\Sigma\Delta}=C[\calZ, 1/\det(\calZ)]$ and $\R\otimes_D \R$ is generated by $C[\calZ, 1/\det(\calZ)]$ as a $\R$-module. By Lemma~\ref{lm:correspondence}, in order to show that $\ker(\fraki)=\{0\}$, it suffices to show that $\ker(\fraki)\cap C[\calZ,1/\det(\calZ)]=\{0\}$. Suppose that $a\in \ker(\fraki) \cap C[\calZ,1/\det(\calZ)]$. Write $a=\sum_{i=1}^m a_i\otimes_D b_i$. Without loss of generality, we may assume that $\{b_1,\dots, b_s\}$ is a $k$-basis of the vector space spanned by $b_1,\dots,b_m$. Let $d\in D$ be nonzero such that $d b_{s+j}=\sum_{i=1}^s c_{ij}b_i$ for some $c_{ij}\in D$ where $j=1,\dots,m-s$. Then we have that
$da=\sum_{i=1}^s (da_i+\sum_{j=1}^{m-s} c_{ij}a_j)\otimes_{D} b_i$. We still have that $\fraki(da)=0$, i.e. $\sum_{i=1}^s (da_i+\sum_{j=1}^{m-s} c_{ij}a_j)\otimes_k b_i=0$. This implies that $da_i+\sum_{j=1}^{m-s} c_{ij}a_j=0$ for all $i=1,\dots,s$, because $b_1,\dots,b_s$ are linearly independent over $k$. Hence $da=0$. Since $a\in C[\calZ,1/\det(\calZ)]\subseteq (\R\otimes_D \R)^{\Sigma\Delta}$, the set $J=\{b\in \R \mid ba=0\}$ is a nonzero $\Sigma\Delta$-ideal. As $\R$ is $\Sigma\Delta$-simple, $1\in J$. In other words, $a=0$ and thus $\ker(\fraki)\cap C[\calZ, 1/\det(\calZ)]=\{0\}$.
\end{proof}
Suppose that $F$ is a field extension of $C$ and $c\in \Hom_C(D,F)$. Then $F$ can be viewed as a $D$-algebra and one can consider $F\otimes_D \R$, where the tensor product is formed using $c$.
\begin{prop}
\label{prop:extensiontorsors}
 Suppose that $T=F\otimes_D \R$ is not the zero ring. Then
 \begin{align*}
    \varphi_T: T\otimes_F T &\longrightarrow T\otimes_C C[G]\\
     a\otimes_F b &\longmapsto (a\otimes_C 1)b(\tilde{\calX} \otimes_C \calZ)
\end{align*} is $T$-isomorphic, where $\tilde{\calX}=1\otimes_D \calX$.
\end{prop}
\begin{proof}
Let $\varphi: R\otimes_k R \rightarrow R\otimes_C C[G]$ be the isomorphism given in (\ref{eqn:varphi}) with $k^{\Sigma\Delta}=C$ and $S=C[G]$. It is easy to verify that the image of $\varphi\circ \fraki$ is $\R\otimes_C C[G]$, where $\fraki$ is given in Lemma~\ref{lm:ringinjectivity}.
 Then we have the following isomorphism:
\[
\xymatrix@R=3pt @!C=3cm{
   (F\otimes_D \R)\otimes_F (F\otimes_D\R) \ar[r] &F\otimes_D \R\otimes_D\R \ar[r]^{1\otimes \varphi\circ \fraki} & F\otimes_D \R\otimes_C C[G]\\
      a(\tilde{\calX}) \otimes_F m(\tilde{\calX})\ar@{|->}[r] & a(\tilde{\calX}) \otimes_D m(\calX) \ar@{|->}[r]& (a(\tilde{\calX})\otimes_C 1)m(\tilde{\calX}\otimes_C \calZ)
}
\]
where $a(\tilde{\calX})\in F\otimes_D \R$ and $m(\tilde{\calX})$ is a monomial in $\tilde{\calX}$.
The first isomorphism is natural (see page 624 of \cite{Lang:algebra}) and the second one is induced by $\varphi\circ \fraki$:
\begin{align*}
    a(\tilde{\calX}) \otimes_D m(\calX)& =(a(\tilde{\calX}) \otimes_D 1) m(1\otimes_D 1 \otimes_D \calX)\\
    &\longrightarrow (a(\tilde{\calX})\otimes_C 1)m(1\otimes_D \calX \otimes_C \calZ)=(a(\tilde{\calX})\otimes_C 1)m(\tilde{\calX} \otimes_C \calZ)
\end{align*}
\end{proof}

The following lemma implies that if $D$ is finitely generated over $C$ then the set of $c\in \Hom_C(D,F)$ such that $F\otimes_D \R$ is not a zero ring is a nonempty Zariski open subset of $\Hom_C(D,F)$. The proof follows from that of Lemma 2.15 of \cite{Feng-Wibmer:differentialgaloisgropus}.
\begin{lem}
\label{lm:zerorings}
 Suppose that $F$ is a field extension of $C$. There is a nonzero $a\in D$ such that for any $c\in \Hom_C(D,F)$ with $c(a)\neq 0$, $F\otimes_D\R$ is not the zero ring.
\end{lem}
\begin{proof}
Consider $c$ as a homomorphism from $D$ to $\bar{F}$, the algebraic closure of $F$.
By Corollary 3 in Section 3.1, Chapter V of \cite{bourbakicommutativealgebra}, there exists a nonzero $a \in D$ such that if $c(a)\neq 0$ then there exists a homomorphism $h$ from $\R$ to $\bar{F}$ such that $c=h|_{D}$. Now suppose that $c(a)\neq 0$ and $h$ is the extension of $c$ to $\R$. Then we have the homomorphism $F\otimes_D\R\rightarrow \bar{F}$ given by $b_1\otimes b_2 \mapsto b_1 h(b_2)$. Since $\bar{F}$ is not the zero ring, so is $F\otimes_D\R$.
\end{proof}

\begin{rem}
\label{rem:specializations}
Lemma~\ref{lm:zerorings} does not provide an explicit $a\in D$. We may find the required homomorphisms $c$ as follows. Write $R=k[X,1/\det(X)]/\frakq$ where $\frakq$ is a maximal $\Sigma\Delta$-ideal. Let $\tilde{\frakq}=\frakq\cap D[X,1/\det(X)]$. Then $\R\cong D[X,1/\det(X)]/\tilde{\frakq}$ and
$$
  F\otimes_D \R\cong F[X,1/\det(X)]/\langle \tilde{\frakq}^c \rangle
$$
where $\tilde{\frakq}^c=\{ P^c \mid \forall\,\,P\in \tilde{\frakq}\}$ and $\langle \tilde{\frakq}^c \rangle$ denotes the ideal in $F[X,1/\det(X)]$ generated by $\tilde{\frakq}^c$.
Therefore $F\otimes_D\R$ is not the zero ring if and only if $\langle \tilde{\frakq}^c \rangle\neq \langle 1 \rangle$.
\end{rem}

Remark that $T=F\otimes_D \R$ inherits the structure of $F$, \ie if we endow $F$ with a differential or difference structure then $T$ will become a differential or difference ring respectively. In the following, we assume that $\tilde{\Sigma}\subseteq \Sigma, \tilde{\Delta}\subseteq \Delta$ and $F$ is a $\tilde{\Sigma}\tilde{\Delta}$-field with algebraically closed field of constants $C=F^{\tilde{\Sigma}\tilde{\Delta}}$. Let $c: D\rightarrow F$ be a $C$-$\tilde{\Sigma}\tilde{\Delta}$-homomoprhism such that $T$ is not the zero ring. Then $T$ is a $\tilde{\Sigma}\tilde{\Delta}$-ring. Let $\frakm$ be a maximal $\tilde{\Sigma}\tilde{\Delta}$-ideal of $T$. Then $T/\frakm$ is a simple $\tilde{\Sigma}\tilde{\Delta}$-ring. Due to Proposition~\ref{prop:extensiontorsors}, for every $g\in G=\Hom_C(C[G],C)$, the map
\begin{equation}
\label{eqn:automorphism}
\rho_g: \xymatrix@R=3pt{T \ar[r] & T \otimes_k T \ar[r] & T \otimes_C C[G] \ar[r] & T \\
                             a(\tilde{\calX}) \ar@{|->}[r] &   1 \otimes_k a(\tilde{\calX})  \ar@{|->}[r] & a(\tilde{\calX}\otimes_C \calZ) \ar@{|->}[r] & a(\tilde{\calX} g(\calZ))}
\end{equation}
 is an $F$-automorphism, where $\tilde{\calX}=1\otimes_D \calX$. Let $I$ be a $\tilde{\Sigma}\tilde{\Delta}$-ideal of $T$. Denote
\begin{equation*}
\label{eqn:stabilizer}
   \stab(\frakm,I)=\{g\in G \mid \rho_g(I)\subseteq \frakm\}.
\end{equation*}
If $I=\frakm$ then we abbreviate $\stab(\frakm, I)$ as $\stab(\frakm)$.
It is clear that $\stab(\frakm)$ is a subgroup of $G$.
\begin{prop}
\label{prop:constantideal}
Let the notation be as above.
\begin{enumerate}
    \item
Suppose $I$ is a $\tilde{\Sigma}\tilde{\Delta}$-ideal of $T$. Then there is an ideal $\Phi_{\frakm,I}$ of $C[G]$ such that the following map is $T/\frakm$-$\tilde{\Sigma}\tilde{\Delta}$-isomorphic:
\begin{equation*}
\bar{\varphi}_T:
\xymatrix@R=3pt{
       T/\frakm \otimes T/I \ar[r] &  T/\frakm \otimes C[G]/\Phi_{\frakm,I}\\
                 a(\bar{\calX}_\frakm)\otimes b(\bar{\calX}_I)  \ar@{|->}[r] &  (a(\bar{\calX}_\frakm)\otimes 1)b(\bar{\calX}_\frakm\otimes \bar{\calZ})
     }
\end{equation*}
 where $\bar{\calX}_\frakm= 1\otimes \calX \mod \frakm, \bar{\calX}_I=1\otimes \calX \mod I$ and $\bar{\calZ}=\calZ \mod \Phi_{\frakm,I}$.
\item  $\stab(\frakm,I)=\{g\in G \mid g(P)=0 \,\,\forall\,\,P\in \Phi_{\frakm,I}\}.$ Consequently, $\stab(\frakm)$ is an algebraic subgroup of $G$.
\end{enumerate}
\end{prop}
\begin{proof}
1. Let $\varphi_T$ be the isomorphism given in Proposition~\ref{prop:extensiontorsors}. Since $\varphi_T(\frakm\otimes T)=\frakm \otimes C[G]$, $\varphi_T$ induces the $T/\frakm$-$\tilde{\Sigma}\tilde{\Delta}$-isomorphism $\tilde{\varphi}_T: T/\frakm \otimes T \rightarrow T/\frakm \otimes C[G]$ which sends $a(\bar{\calX}_\frakm)\otimes b(\calX)$ to $(a(\bar{\calX}_\frakm)\otimes 1)b(\bar{\calX}_\frakm\otimes \calZ)$. Therefore, it suffices to show that there is an ideal $\Phi_{\frakm,I}$ of $C[G]$ such that $\tilde{\varphi}_T(T/\frakm \otimes I)=T/\frakm \otimes \Phi_{\frakm,I}$. By Proposition 6.14 of \cite{hardouin2008differential} with $\Pi=\emptyset$, one has that $(T/\frakm)^{\tilde{\Sigma}\tilde{\Delta}}=F^{\tilde{\Sigma}\tilde{\Delta}}=C$. Hence $(T/\frakm\otimes C[G])^{\tilde{\Sigma}\tilde{\Delta}}=1\otimes C[G]=C[G]$. Note that $T/\frakm$ can be viewed as a subring of $T/\frakm\otimes C[G]$ and moreover $T/\frakm\otimes C[G]$ is generated by $(T/\frakm\otimes C[G])^{\tilde{\Sigma}\tilde{\Delta}}$ as a $T/\frakm$-module. Set
$$
  \Phi_{\frakm,I}=\tilde{\varphi}_T(T/\frakm\otimes I)\cap C[G].
$$
By Lemma~\ref{lm:correspondence}, the $\tilde{\Sigma}\tilde{\Delta}$-ideal $\tilde{\varphi}_T(T/\frakm\otimes I)$ is generated by $\Phi_{\frakm,I}$.
It is clear that the ideal in $T/\frakm \otimes C[G]$ generated by $\Phi_{\frakm,I}$ is $T/\frakm\otimes \Phi_{\frakm,I}$. Hence we have that $\tilde{\varphi}_T(T/\frakm\otimes I)=T/\frakm\otimes \Phi_{\frakm,I}$ as desired.

2.  Set $H=\{g\in G \mid g(P)=0 \,\,\forall\,\,P\in \Phi_{\frakm,I}\}$.
Let $\{a_i \mid i\in \I_1\}$ be a $C$-basis of $\frakm$ and let $\{a_i\mid i\in \I_1\cup \I_2\}$ be a $C$-basis of $T$. Suppose that $g\in H$ and $b\in I$. From the statement 1, one sees that $\varphi_T(\frakm\otimes T+T\otimes I)=\frakm\otimes C[G]+T\otimes \Phi_{\frakm,I}$. Hence we may write
$$
   \varphi_T(1\otimes b)=b(\tilde{\calX} \otimes \calZ)=\sum_{i\in \I_1\cup \I_2} a_i\otimes \beta_i
$$
where $\beta_i\in C[G]$ and moreover $\beta_i\in \Phi_{\frakm,I}$ if $i\in \I_2$. Using (\ref{eqn:automorphism}), one sees that
$$
   \rho_g(b)=b(\tilde{\calX} g(\calZ))=\sum_{i\in \I_1\cup \I_2} a_i\beta_i(g(\calZ))=\sum_{i\in \I_1\cup\I_2} a_ig(\beta_i)=\sum_{i\in \I_1} a_ig(\beta_i)\in \frakm.
$$
Thus $\rho_g(I)\subseteq \frakm$. In other words, $g\in \stab(\frakm,I)$. On the other hand, suppose $g\in \stab(\frakm,I)$. Let $\beta\in \Phi_{\frakm,I}$. Then there is $b\in \frakm\otimes T+T\otimes I$ such that $\varphi_T(b)=1\otimes \beta$.
Write $b=\tilde{b}+\sum_{i\in \I_2} a_i\otimes b_i$ where $\tilde{b}\in \frakm\otimes T$ and $b_i\in I$. Since $\varphi_T(\tilde{b})\in\frakm\otimes C[G]$, one sees that
$$
  \varphi_T\left(\sum_{i\in \I_2}a_i\otimes b_i\right)=\varphi_T(b-\tilde{b})=1\otimes \beta+\sum_{i\in \I_1} a_i \otimes \beta_i
$$
for some $\beta_i\in C[G]$. Using (\ref{eqn:automorphism}) again, one has that
$$
  \sum_{i\in \I_2}a_i\rho_g(b_i)=g(\beta)+\sum_{i\in \I_1}a_ig(\beta_i).
$$
 As $\rho_g(b_i)\in \frakm$, $g(\beta)+\sum_{i\in \I_1}a_ig(\beta_i)\in \frakm$. Hence $g(\beta)=0$. In othe words, $g\in H$. Thus $\stab(\frakm,I)=H$.
\end{proof}

As a corollary, we have the following:
\begin{cor}
\label{cor:maximaltomaximal}
Suppose that $\frakm'$ is another maximal $\tilde{\Sigma}\tilde{\Delta}$-ideal of $T$. Then there exists $g\in G$ such that $\rho_g(\frakm')=\frakm$. In this case $\stab(\frakm)$ is conjugate to $\stab(\frakm')$ by $g$.
\end{cor}
\section{A condition for a $\sigma\delta$-Picard-Vessiot ring to be $\sigma$-Picard-Vessiot}
\label{sec:pvrings}
In the remainder of this paper, we will focus on $\Sigma\Delta$-rings with at most one single automorphism $\sigma$ and at most one single derivative $\delta$. When $\Sigma$ and $\Delta$ are specified, we shall use the prefixes $\sigma$-, $\delta$-, $\sigma\delta$- and the superscripts $(\cdot)^\sigma,(\cdot)^\delta,(\cdot)^{\sigma\delta}$ instead of $\Sigma\Delta$- or $(\cdot)^{\Sigma\Delta}$. 

Throughout this section, let $k_0$ be a $\delta$-field with algebraically closed field of constants $C=k_0^\delta$ and let $k_0(x)$ is the $\sigma\delta$-field with $\sigma(x)=x+1$. We consider the following $\sigma\delta$-linear system over $k_0(x)$:
\begin{equation}
\label{eqn:sigmadelta-eqn}
\sigma(Y)=AY,\,\,\delta(Y)=BY
\end{equation}
where $A\in \GL_n(k_0(x)), B\in \gl_n(k_0(x))$ and $A,B$ satisfy the integrability condition: $\sigma(B)A=\delta(A)+AB$.
\begin{notation}
\label{not:sufficientcondition}
Throughout this section, we further assume
\begin{itemize}
\item $R$ is a $\sigma\delta$-Picard-Vessiot ring over $k_0(x)$ for (\ref{eqn:sigmadelta-eqn}). 
\item $\calX$ is a fixed fundamental solution matrix in $\GL_n(R)$.
\item
$K$ is a $\delta$-Picard-Vessiot extension field of $k_0$ for $\delta(Y)=B(c)Y$ for some $c\in C$, where $B(c)$ denotes replacing $x$ with $c$ in $B$.
\item
$\hat{R}$ is a $\sigma\delta$-Picard-Vessiot ring over $K(x)$ for (\ref{eqn:sigmadelta-eqn}) containing $R$.
\item
$\K,\hat{\K}$ are the total rings of fractions of $R,\hat{R}$ respectively.
\item
$\breve{R}$ is the composite of $R$ and $\K^\sigma (x)$ inside $\K$.
\end{itemize}
\end{notation}
Note that $\hat{R}$ always exists. For instance, let $\frakm$ be a maximal $\sigma\delta$-ideal of $K(x)\otimes_{k_0(x)}R$. Then $(K(x)\otimes_{k_0(x)}R)/\frakm$ is a $\sigma\delta$-Picard-Vessiot ring over $K(x)$ for (\ref{eqn:sigmadelta-eqn}). Since $R$ is $\sigma\delta$-simple, the $\sigma\delta$-homomorphism $R\rightarrow (K(x)\otimes_{k_0(x)}R)/\frakm, a\mapsto \overline{1\otimes a}$ is injective and thus we may consider $R$ as a subring of $\hat{R}$.

If $R$ is a simple $\sigma$-ring then it will be a $\sigma$-Picard-Vessiot ring for $\sigma(Y)=AY$. However $R$ is generally not $\sigma$-simple. In this section, we shall show that $\breve{R}$ is a simple $\sigma$-ring and thus it is a $\sigma$-Picard-Vessiot ring for $\sigma(Y)=AY$ over $\K^\sigma(x)$.
We start with the following lemma.
\begin{lem}
\label{lm:sigmaconstant}
$\K^\sigma$ is a $\delta$-field.
\end{lem}
\begin{proof}
We first show that any nonzero element of $\K^\sigma$ is not a zero divisor of $\K$. Suppose that $a\in \K^\sigma\setminus\{0\}$ and $a$ is a zero divisor of $\K$, \ie there is a nonzero $b\in \K$ such that $ab=0$. Write $a=p_1/q_1, b=p_2/q_2$ where $p_i,q_i\in R$ and neither of $q_1,q_2$ is a zero divisor. Then $p_1$ is a zero divisor. Lemma 19 of \cite{feng2010liouvillian} implies that there is a positive integer $s$ such that $\prod_{i=1}^s \sigma^i(p_1)=0$. This implies that $\prod_{i=1}^s \sigma^i(a)=0$. Since $\sigma(a)=a$, $a^{s+1}=0$. In other words, $p_1^s=0$. Since $R$ is reduced, $p_1=0$ and thus $a=0$, a contradiction. Now for each $a\in \K^\sigma\setminus\{0\}$, since $a$ is not a zero divisor of $\K$, there is $b\in \K$ such that $ab=1$. It is clear that $\sigma(b)=b$, \ie $b\in \K^\sigma$. So $\K^\sigma$ is a field.
\end{proof}

It was shown in \cite{van2006galois} as well as \cite{Wibmer:SMLtheorems} that every $\sigma$-Picard Vessiot ring over $F(x)$ can be embedded into the ring of sequences $\Seq_F$, where $F$ is a field. We shall first show that the $\sigma\delta$-Picard-Vessiot ring $R$ can be embedded into the ring of sequences $\Seq_K$. 

In the following, we fix a $c\in C$ such that for each $i\in\mathbb{Z}$, $A(c+i), B(c+i)$ are well-defined and $\det(A(c+i))\neq 0$. Let $K$ be a $\delta$-Picard-Vessiot extension field of $k_0$ for $\delta(Y)=B(c)Y$.

\begin{rem}
\label{rm:deltapv}
For each $i\in \mathbb{Z}$, $K$ is also a $\delta$-Picard-Vessiot extension field of $k_0$ for $\delta(Y)=B(c+i)Y$. Since $\sigma(B)A=\delta(A)+AB$, one has that
$$
   B(c+1)A(c)=\delta(A(c))+A(c)B(c).
$$
By induction, one can verify that for each $s>0$,
$$
  B(c+s)\prod_{i=1}^{s} A(c+s-i)=\delta\left(\prod_{i=1}^{s} A(c+s-i)\right)+\left(\prod_{i=1}^{s} A(c+s-i)\right) B(c).
$$
Similarly, since $BA^{-1}=\delta(A^{-1})+A^{-1}\sigma(B)$, one has that for each $s<0$,
$$
  B(c+s)\prod_{i=s}^1 A^{-1}(c+i)=\delta\left(\prod_{i=s}^{1} A^{-1}(c+i)\right)+\left(\prod_{i=s}^1 A^{-1}(c+i)\right) B(c).
$$
As $\det(\prod_{i=1}^{s} A(c+s-i))\neq 0$ and $\det(\prod_{i=s}^1 A^{-1}(c+i))\neq 0$, the above two equalities imply that the systems $\delta(Y)=B(c)Y$ and $\delta(Y)=B(c+s)Y$ are equivalent over $k_0$. Hence they have the same $\delta$-Picard-Vessiot  extension fields.
\end{rem}

The ring of sequences $\Seq_K$ is defined to be the set
$$
   \Seq_K=\{(c_0,c_1,\dots)\mid c_i\in K\}/\sim
$$
where $(b_0,b_1,\dots)\sim (c_0,c_1,\dots)$ if there is a nonnegative integer $d$ such that $b_i=c_i$ for all $i\geq d$.
We may endow $\Seq_K$ with a $\sigma\delta$-ring structure by setting
\begin{align*}
  \delta((c_0,c_1,\dots))=(\delta(c_0),\delta(c_1),\dots),\,\,
  \sigma((c_0,c_1,\dots))=(c_1,c_2,\dots).
\end{align*}
By sending $a\in K$ into $(a,a,\dots)$, we can embed $K$ into $\Seq_K$ and moreover this embedding map is a $\delta$-homomorphism. Furthermore, we can embed $K(x)$ into $\Seq_K$ by sending $f(x)\in K(x)$ to
$$
    (0,\dots,0,f(c+\nu_f),f(c+\nu_f+1),\dots)
$$
where $\nu_f$ is a nonnegative integer such that $f(x)$ is well-defined at $x=c+i$ for all $i\geq \nu_f$. One can verify that this embedding map is a $\sigma\delta$-homomorphism. Under this embedding map, we may consider $K(x)$ as a subring of $\Seq_K$.
\begin{prop}
\label{prop:embeddingmap}
 $R$ can be embedded over $k_0(x)$ into $\Seq_K$.
\end{prop}
\begin{proof}
It suffices to show that $\hat{R}$ can be $K(x)$-embedded into $\Seq_K$.  To this end, we only need to show that there is a $\sigma\delta$-Picard-Vessiot ring over $K(x)$ for (\ref{eqn:sigmadelta-eqn}) inside $\Seq_K$, because all $\sigma\delta$-Picard-Vessiot rings over $K(x)$ for (\ref{eqn:sigmadelta-eqn}) are isomorphic.  By Proposition 2.4 of \cite{Wibmer:SMLtheorems}, there is a $\sigma$-Picard-Vessiot ring over $K(x)$ for $\sigma(Y)=AY$ inside $\Seq_K$. Precisely, set $W=(W_0, W_1,\dots)$ with
$$
   W_0=I_n,  W_s=A(c+s-1)W_{s-1}, \forall s\geq 1.
$$
Then $K(x)[W,\frac{1}{\det(W)}]$ is a $\sigma$-Picard-Vessiot ring over $K(x)$ for $\sigma(Y)=AY$.
Let $U\in {\rm GL}_n(K)$ be a fundamental matrix of $\delta(Y)=B(c)Y$.
We then have that $\sigma(WU)=\sigma(W)U=AWU$. Moreover, for each $s\geq 0$,
\begin{align*}
\delta(W_s U)&=\delta(W_s)U+W_s\delta(U)\\
       &=(\delta(W_s)W_s^{-1}+W_s B(c)W_s^{-1})W_sU=B(c+s)W_sU.
\end{align*}
The last equality holds because of Remark~\ref{rm:deltapv}. Hence $\delta(WU)=BWU$. In other words, $WU$ is a fundamental solution matrix of (\ref{eqn:sigmadelta-eqn}). Note that
$K(x)[W,\frac{1}{\det(W)}]=K(x)[WU, \frac{1}{\det(WU)}]$. Thus $K(x)[WU,\frac{1}{\det(WU)}]$ is a $\sigma\delta$-ring and moreover since it is $\sigma$-simple, it is $\sigma\delta$-simple. Consequently, $K(x)[WU,\frac{1}{\det(WU)}]$ is a $\sigma\delta$-Picard-Vessiot ring over $K(x)$ for (\ref{eqn:sigmadelta-eqn}).
\end{proof}

\begin{exam}
\label{exam:sigmadeltaPV}
We shall construct a $\sigma\delta$-Picard-Vessiot ring over $C(m,t)$ for (\ref{eqn:Tchebychev}).
Consider the following linear differential equation
\begin{equation*}
\label{EQ19}
\delta(Y)=\begin{pmatrix}
-\frac{t}{1-t^2}&\frac{1}{1-t^2}\\
0&0
\end{pmatrix}Y.
\end{equation*}
We have that
$$
    U=\begin{pmatrix} t+\sqrt{t^2-1} & t-\sqrt{t^2-1} \\ 1 & 1 \end{pmatrix}
$$
is a fundamental solution matrix and thus $K=C(t,\sqrt{t^2-1})$ is a $\delta$-Picard-Vessiot exntesion of $C(t)$ for the above equation. Set $A=\begin{pmatrix} 0 & 1 \\ -1 & 2t \end{pmatrix}$ and
$$
    W=\left(U,AU, A^2U,A^3U,\dots\right).
$$
Then $C(m,t)[W, 1/\det(W)]$ is a $\sigma\delta$-Picard-Vessiot ring over $C(m,t)$ for (\ref{eqn:Tchebychev}). We claim that $C(m,t)[W,1/\det(W)]$ is an integral domain. An easy calculation yields that $d=U^{-1}AU=\diag(t-\sqrt{t^2-1}, t+\sqrt{t^2-1})$.  Therefore
$$
     W=Ud^m =\begin{pmatrix} (t-\sqrt{t^2-1})^{m-1} & (t+\sqrt{t^2-1})^{m-1} \\ (t-\sqrt{t^2-1})^{m} & (t+\sqrt{t^2-1})^{m}\end{pmatrix}.
$$
Let $\eta=(t+\sqrt{t^2-1})^{m-1}$. Then $C(m,t)[W,1/\det(W)]=K(m)[\eta,\frac{1}{\eta}]$. Due to Corollary 2.4 of \cite{LiWuZheng:testinglineardependence}, $\eta$ is transcendental over $K(m)$. This implies that $K(m)[\eta,\frac{1}{\eta}]$ is an integral domain.
\end{exam}

\begin{exam}
\label{exam:sigmadeltagroups}
Let us compute the corresponding $\sigma\delta$-Galois group of the system (\ref{eqn:Tchebychev}). In this example as well as the examples in the remainder of this paper, we always embed the Galois group into $\GL_2(C)$.
Consider the $C(m,t)$-$\sigma\delta$-homomorphism
\begin{align*}
   \varphi: C(m,t)[X,1/\det(X)] &\longrightarrow C(m,t)[W,1/\det(W)]\\
                           f(X)&\longmapsto f(W)
\end{align*}
where $W$ is given as in Example~\ref{exam:sigmadeltaPV}.
Let us calculate $\ker(\varphi)$. Set $I=\langle f_1 ,f_2, f_3 \rangle$, where
\begin{equation}
\label{eqn:generators}
   f_1=X_{11}X_{12}-1, f_2=X_{21}X_{22}-1, f_3=(X_{11}X_{22})^2-2tX_{11}X_{22}+1.
\end{equation}
We claim that $\ker(\varphi)=I$. It is easy to verify that $I\subset \ker(\varphi)$. Suppose that $f\in \ker(\varphi)\cap C(m,t)[X]$. There are positive integers $d_1, d_2$ such that
$$
   f\equiv X_{11}^{d_1}X_{22}^{d_2}f\equiv \sum_{i=0}^s (a_{i,1}(X_{11}X_{22})+a_{i,0})X_{11}^i \mod I,
$$
where $a_{i,1},a_{i,0}\in C(m,t)$. Since $f\in \ker(\varphi)$, $\sum_{i=0}^s (a_{i,1}(X_{11}X_{22})+a_{i,0})X_{11}^i\in \ker(\varphi)$. In other words,
$$
  \sum_{i=0}^s (a_{i,1}(t+\sqrt{t^2-1})+a_{i,0})\eta^i=0
$$
where $\eta=(t+\sqrt{t^2-1})^{m-1}$.
As $\eta$ is transcendental over $K$, $a_{i,1}(t+\sqrt{t^2-1})+a_{i,0}=0$ for all $i$. So $a_{i,1}=a_{i,0}=0$ for all $i$. Consequently, $f\in I$ and then $I=\ker(\varphi)$. From this, we have
\begin{align*}
   G&=\{g\in \GL_2(C) \mid \rho_g(P)=P(Xg)\in I, \forall\,\,P\in I\}\\
   &=\{(g_{ij})\in \GL_2(C) \mid g_{11}g_{12}=0, g_{21}g_{22}=0, g_{11}g_{22}+g_{12}g_{21}=1\}.
\end{align*}
\end{exam}

From the proof of Proposition~\ref{prop:embeddingmap}, one can easily see that
\begin{cor}
\label{cor:sigmapv}
$\hat{R}$ is a $\sigma$-Picard-Vessiot ring over $K(x)$ for $\sigma(Y)=AY$ and $\hat{R}^\sigma=K$.
\end{cor}
\begin{proof}
For the first assertion, it suffices to show that $\hat{R}$ is $\sigma$-simple.
The proof of Proposition~\ref{prop:embeddingmap} implies that $\hat{R}$ is $\sigma\delta$-isomorphic to $K(x)[WU, 1/\det(WU)]=K(x)[W,1/\det(W)]$. Since $K(x)[W,1/\det(W)]$ as a $\sigma$-Picard-Vessiot ring for $\delta(Y)=AY$ is $\sigma$-simple, $\hat{R}$ is $\sigma$-simple. The second assertion follows from the fact that $(\Seq_K)^\sigma=K$.
\end{proof}

The following example implies that not all $\sigma\delta$-Picard-Vessiot rings are $\sigma$-simple.
\begin{exam}
Consider the $\sigma\delta$-system
$$
  \sigma(y)=y,\delta(y)=2ty.
$$
Then $C(x,t)[y,1/y]$ is a $\sigma\delta$-Picard-Vessiot ring over $C(x,t)$ for this system. While $C(x,t)[y,1/y]$ is not a $\sigma$-Picard-Vessiot ring over $C(x,t)$ for $\sigma(y)=y$, because the $\sigma$-ideal generated by $y+1$ is nontrivial. However, set $K=C(t,e^{t^2})$ which is a $\delta$-Picard-Vessiot extension field of $C(t)$ for $\delta(y)=2ty$. Then $K(x)$ is a $\sigma\delta$-Picard-Vessiot ring over $K(x)$ for the above $\sigma\delta$-system and it is also a $\sigma$-Picard-Vessiot ring over $K(x)$ for $\sigma(y)=y$.
\end{exam}

In what follows, we shall show that not only $\hat{R}$ but also $\breve{R}$ are $\sigma$-simple.
\begin{lem}
\label{lm:intersection} $\K\cap K(x)=\K^\sigma(x)$.
\end{lem}
\begin{proof}
It suffices to show that $\K\cap K(x)\subset \K^\sigma(x)$. Suppose $f=p/q\in \K\cap K(x)$ where $p,q\in K[x], \gcd(p,q)=1$ and $q$ is monic. Write $q=x^d+\sum_{i=0}^{d-1} a_i x^i$ and $p=\sum_{j=0}^s a_{j+d} x^i$, where $a_i \in K$. Then we only need to show that $a_i \in \K^\sigma$. First of all, we have that $1,x,\dots,x^s, fx^{d-1},\dots,f$ are linearly independent over $K$. Otherwise, there are $\alpha_0, \dots,\alpha_s, \beta_0, \dots,\beta_{d-1}\in K$, not all zero, such that $\sum_{i=0}^s \alpha_i x^i -f \sum_{j=0}^{d-1}\beta_j x^j=0$. Replacing $f$ with $p/q$ yields that $q(\sum_{i=0}^s \alpha_i x^i)=p(\sum_{j=0}^{d-1}\beta_j x^j)=0$. Since $\gcd(p,q)=1$, $q$ divides $\sum_{j=0}^{d-1}\beta_j x^j$. So $\sum_{j=0}^{d-1}\beta_j x^j=0$ and then $\sum_{i=0}^s \alpha_i x^i=0$. In other words, all $\alpha_i$ and $\beta_j$ are zero, a contradiction. Secondly, let $M$ be the matrix formed by $(x+l)^s,\dots,1, \sigma^l(fx)^{d-1}, \dots, \sigma^l(f)$ with $l=0,1,\dots,s+d$. Then $\det(M)\neq 0$. Applying $\sigma^l$ to $\sum_{i=0}^s a_{d+i} x^i -\sum_{j=0}^{d-1} a_j (fx^j)=x^df, l=0,\dots,s+d$ yields that 
$$
   M(a_{s+d},\dots,a_0)^t=(x^df, \dots,\sigma^{s+d}(x^df))^t.
$$
Multiplying the adjoint matrix of $M$ on both sides, we have that $\det(M)a_i\in \K$ for all $i$. Note that both $\K$ and $K(x)$ are in $\hat{\K}$. As $\det(M)$ is invertible in $K(x)$, it is invertible in $\hat{\K}$ and so it is not a zero divisor in $\K$. This implies that $a_i\in \K$ and then $a_i\in \K^\sigma$ as desired.
\end{proof}
\begin{prop}
\label{prop:strongsigmapv}
Let the notation be as in Notation~\ref{not:sufficientcondition}. Then
$\breve{R}$ is a $\sigma$-Picard-Vessiot ring over $\K^\sigma(x)$ for $\sigma(Y)=AY$ and $\breve{R}^\sigma=\K^\sigma$.
\end{prop}
\begin{proof}
Note that $\breve{R}=\K^\sigma(x)[\calX,1/\det(\calX)]$. It suffices to show that $\breve{R}$ is $\sigma$-simple.
We first prove that $\breve{R}$ is $\sigma\delta$-simple. Suppose that $I$ is a $\sigma\delta$-ideal of $\breve{R}$ and $I\neq (0)$. Then $I\cap R\neq (0)$ and it is a $\sigma\delta$-ideal of $R$. Since $R$ is $\sigma\delta$-simple, $1\in I\cap R\subset I$ and thus $I=\breve{R}$.

Now suppose $a\in \breve{R}\setminus\{0\}$. We shall show that the $\sigma$-ideal of $\breve{R}$ generated by $a$ is trivial. Let $\{a_1=a,\dots,a_m\}$ be a basis of $K$-vector space spanned by $\{g(a) \mid \forall\,\,g\in \sigma\delta\mbox{-}\Gal(\hat{R}/K(x))\}$. Note that $\breve{R}$ is invariant under the action of $\sigma\delta$-$\Gal(\hat{R}/K(x))$. All $a_i$ can be chosen to be in $\breve{R}$. By Corollary~\ref{cor:sigmapv}, $\hat{R}$ is $\sigma$-simple and $\hat{R}^\sigma=K$. By Proposition~\ref{prop:lineardependence} with $\Sigma=\{\sigma\}$ and $\Delta=\emptyset$, there are $s_1,\dots,s_m$ with $s_1=0$ such that $d=\det((\sigma^{s_i}(a_j))_{1\leq i,j\leq m})\neq 0$. For each $g\in \sigma\delta$-$\Gal(\hat{R}/K(x))$, one has that $g(d)=c_g d$ with $c_g\in K$. In other words, the ideal $(d)$ of $\hat{R}$ generated by $d$ is a $\sigma\delta$-$\Gal(\hat{R}/K(x))$-ideal. Since $\hat{R}$ is a $\sigma\delta$-Picard-Vessiot ring over $K(x)$ for (\ref{eqn:sigmadelta-eqn}), Corollary~\ref{cor:groupinvariantideal} implies that $d$ is invertible in $\hat{R}$. Now one has that both $\sigma(d)d^{-1}$ and $\delta(d)d^{-1}$ are invariant under the action of $\sigma\delta$-$\Gal(\hat{R}/K(x))$. The Galois correspondence implies that $\sigma(d)d^{-1}, \delta(d)d^{-1}\in K(x)$. Set $b_1=\sigma(d)d^{-1}, b_2=\delta(d)d^{-1}$. Since $d$ is not a zero divisor in $\hat{R}$ and $d\in \breve{R}$, $d$ is not a zero divisor in $\breve{R}$. Therefore $b_1, b_2\in \K$ because $\K$ is also the total ring of fractions of $\breve{R}$. This implies that $b_1,b_2\in \K\cap K(x)$. By Lemma~\ref{lm:intersection}, $\K\cap K(x)=\K^\sigma(x)$. Thus the ideal $(d)$ of $\breve{R}$ generated by $d$ is a $\sigma\delta$-ideal. As $\breve{R}$ is $\sigma\delta$-simple, $d$ is invertible in $\breve{R}$. Expanding $d$ by the first column, one sees that $d$ belongs to the $\sigma$-ideal of $\breve{R}$ generated by $a$ and so this ideal is trivial.
The second assertion is obvious.
\end{proof}


Since $\breve{R}^\sigma=\K^\sigma=(\K^\sigma(x))^\sigma$, by Proposition~\ref{prop:galoisgroups} or \cite{chatzidakis2007definitions}, $\sigma$-$\Gal(\breve{R}/\K^\sigma(x))$ can be identified with $\Hom_{\K^\sigma}((\breve{R}\otimes_{\K^\sigma(x)} \breve{R})^\sigma, \K^\sigma)$. As usual, for $H=\Hom_C(D,C)$ and a $C$-algebra $S$, denote by $H(S)$ the set of $S$-points of $H$, \ie $H(S)=\Hom_C(D,S)=\Hom_S(S\otimes_C D,S)$.
\begin{lem}
\label{lm:connections}
Let $H=\Hom_C((\breve{R}\otimes_{\K^\sigma(x)} \breve{R})^{\sigma\delta}, C)$. Then $\sgal(\breve{R}/\K^\sigma(x))$ can be identified with $H(\K^\sigma)$.
\end{lem}
\begin{proof}
 By Proposition~\ref{prop:strongsigmapv}, one has that $\breve{R}^\sigma=\K^\sigma$. Since $\breve{R}$ is $\sigma$-simple, one sees that $(\breve{R}\otimes_{\K^\sigma(x)} \breve{R})^{\sigma}=\K^\sigma[\calZ,1/\det(\calZ)]$, where $\calZ=\calX^{-1}\otimes_{\K^\sigma(x)} \calX$. As the total ring of fractions of a $\sigma\delta$-simple ring, it is easy to see that $\K$ is $\sigma\delta$-simple. So $\K^\sigma$ is $\delta$-simple. By Lemma~\ref{lm:linearlydisjoint} with $S=\K^\sigma, T=\K^\sigma[\calZ,1/\det(\calZ)]$ and $\Sigma=\emptyset, \Delta=\{\delta\}$, $\K^\sigma$ and $C[\calZ,1/\det(\calZ)]$ are linearly disjoint over $(\K^{\sigma})^{\delta}=C$. Hence the natural homomorphism $\K^{\sigma}\otimes_C C[\calZ,1/\det(\calZ)]\rightarrow \K^\sigma[\calZ,1/\det(\calZ)], a\otimes b\mapsto ab$ is isomorphic. Since $(\breve{R}\otimes_{\K^\sigma(x)} \breve{R})^{\sigma\delta}=C[\calZ,1/\det(\calZ)]$, one sees that 
 \begin{align*}H(\K^\sigma)&=\Hom_{C}((\breve{R}\otimes_{\K^\sigma(x)} \breve{R})^{\sigma\delta}, \K^\sigma)=\Hom_{C}(C[\calZ,1/\det(\calZ)], \K^\sigma)\\
 &=\Hom_{\K^\sigma}(\K^\sigma\otimes_C C[\calZ,1/\det(\calZ)],\K^\sigma)\\
 &=\Hom_{\K^\sigma}(\K^\sigma[\calZ,1/\det(\calZ)],\K^\sigma)=\Hom_{\K^\sigma}((\breve{R}\otimes_{\K^\sigma(x)} \breve{R})^\sigma,\K^\sigma).
 \end{align*} The lemma then follows from Proposition~\ref{prop:galoisgroups}.
\end{proof}

\section{Main results}
\label{sec:mainresults}
In this section, we shall present the main results of this paper. We assume that 
\begin{itemize}
  \item $\calD$ is a simple $\delta$-domain that is finitely generated over $C=\calD^{\delta}$. For example, $\calD=C[t]$,
  \item $k_0$ is the field of fractions of $\calD$ (thus $k_0^\delta=\calD^\delta=C$),
  \item $R$ is a $\sigma\delta$-Picard-Vessiot ring over $k_0(x)$ for (\ref{eqn:sigmadelta-eqn}),
  \item $\calX$ is a fundamental solution matrix of (\ref{eqn:sigmadelta-eqn}) in $\GL_n(R)$,
  \item $\frakh\in \calD[x]\setminus \{0\}$ satisfies that all entries of $A$ and $B$ are in $\calD[x,1/\frakh]$,
  \item
$D= \calD[x][\{\frac{1}{\sigma^i(\frakh)} \mid \forall\,\,i\in \mathbb{Z}\}]$,
\item
$\R=D[\calX,1/\det(\calX)]$.
  \end{itemize}
It is clear that $D$ is a $\sigma\delta$-ring. Furthermore, we shall show that $D$ is actually a simple $\sigma\delta$-ring and so the results presented in Section~\ref{sec:specializations} can be applied.
\begin{lem}
\label{lm:sigmadeltasimple}
$D$ is a simple $\sigma\delta$-ring.
\end{lem}
\begin{proof}
 Let $I$ be a nonzero $\sigma\delta$-ideal of $D$. Then there is a nonzero $p\in I\cap \calD[x]$. Let $s$ be a positive integer such that $p$ and $\sigma^s(p)$ viewed as polynomials in $x$ have no common roots. Then there are $a,b\in \calD[x]$ such that $ap+b\sigma^s(p)\in (\calD\cap I)\setminus \{0\}$. Since $\calD$ is $\delta$-simple, $1\in I$.
\end{proof}

\subsection{Galois groups of the specializations}
\label{subsec:algsubgroups}
In this subsection, we shall show that the Galois groups of the specializations of the linear difference equation and the linear differential equation in (\ref{eqn:sigmadelta-eqn}) are algebraic subgroups of $G$.

 Let $c_1\in \Hom_C(\calD, C)$. Then $c_1$ lifts to a unique element in $\Hom_{C[x]}(D,C(x))$ whose restrict on $\calD$ is equal to $c_1$. As before, we still use $c_1$ to denote its lifting. Suppose that $c_1(\frakh)\neq 0$. As $c_1$ is a $\sigma$-homomorphism, $C(x)\otimes_D \R$ is a $\sigma$-ring if it is not the zero ring. Suppose that $C(x)\otimes_D \R$ is not the zero ring. Let $\frakm$ be a maximal $\sigma$-ideal of $C(x)\otimes_D \R$. Further assume that $\det(A^{c_1})\neq 0$, where $(\cdot)^{c_1}$ denotes the application of $c_1$ to the entries of the corresponding matrix. Then $(C(x)\otimes_D \R)/\frakm$ is a $\sigma$-Picard-Vessiot ring over $C(x)$ for the following system
\begin{equation*}
       \sigma(Y)=A^{c_1} Y .
\end{equation*}
\begin{notation}
\label{not:sigma}
Set $\calS_{\sigma,c_1}=(C(x)\otimes_D \R)/\frakm$.
\end{notation}

Due to Proposition~\ref{prop:constantideal} with $F=C(x), \tilde{\Sigma}=\{\sigma\}$ and $\tilde{\Delta}=\emptyset$, one has that
\begin{align*}
  \stab(\frakm)&=\{g\in G \mid g(P)=0\,\,\forall\,P\in \Phi_{\frakm,\frakm}\}=\Hom_C(C[G]/\Phi_{\frakm,\frakm},C).
\end{align*}
where $\Phi_{\frakm,\frakm}$ is given as in Proposition~\ref{prop:constantideal}. As $\stab(\frakm)$ is an algebraic subgroup of $G$ by Proposition~\ref{prop:constantideal}, $C[G]/\Phi_{\frakm,\frakm}$ is a Hopf algebra. 
Using Proposition~\ref{prop:galoisgroups}, we immediately have the following theorem.
\begin{thm}
\label{thm:sigmasubgroups}
Let $c_1\in \Hom_C(\calD,C)$ be such that $c_1(\frakh)\neq 0$, $\det(A^{c_1})\neq 0$ and $C(x)\otimes_D \R$ is not the zero ring. Suppose that $\frakm$ is a maximal $\sigma$-ideal of $C(x)\otimes_D \R$. Then
$\stab(\frakm)$ is the $\sigma$-Galois group of $\calS_{\sigma,c_1}$ over $C(x)$.
\end{thm}
\begin{proof}
Due to Proposition~\ref{prop:constantideal} with $F=C(x), \tilde{\Sigma}=\{\sigma\}$, $\tilde{\Delta}=\emptyset$ and $I=\frakm$, one sees that $(\calS_{\sigma,c_1}\otimes_{C(x)} \calS_{\sigma,c_1})^\sigma$ is isomorphic to $C[G]/\Phi_{\frakm,\frakm}$. Furthermore, one can verify that they are isomorphic as Hopf algebras. By Proposition~\ref{prop:galoisgroups}, $\sgal(\calS_{\sigma,c_1}/C(x))$ can be identified with $\Hom_C(\calS_{\sigma,c_1}\otimes_{C(x)} \calS_{\sigma,c_1})^\sigma, C)$ and thus with 
$ \Hom_C(C[G]/\Phi_{\frakm,\frakm},C)=\stab(\frakm). $
\end{proof}
\begin{rem}
\label{rem:sigmagroups}
Note that different choices of maximal $\sigma$-ideals $\frakm$ may lead to different algebraic groups $stab(\frakm)$ of $G$. Corollary~\ref{cor:maximaltomaximal} implies that these $\stab(\frakm)$ are conjugate by elements of $G$.
\end{rem}

Similarly, let $c_2\in \Hom_C(C[x],C)$ be such that $\frakh(c_2(x)+i)\neq 0$ for all $i\in \Z$. We have that $c_2$ lifts to a unique $\tilde{c}_2\in \Hom_{\calD}(D,k_0)$ such that $\tilde{c}_2|_{C[x]}=c_2$. Again, for the sake of notation, we still use $c_2$ to denote $\tilde{c}_2$. Since $c_2(a)=a$ for any $a\in \calD$, one sees that $c_2$ is a $\delta$-homomorphism. Then $k_0\otimes_D \R$ is a $\delta$-ring if it is not the zero ring. Suppose that $k_0\otimes_D \R$ is not the zero ring. Let $\frakn$ be a maximal $\delta$-ideal of $k_0\otimes_D \R$. Then $(k_0 \otimes_D \R)/\frakn$ is a $\delta$-Picard-Vessiot ring over $k_0$ for the following system
\begin{equation*}
\label{eqn:specializationofdifferencepart}
       \delta(Y)=B^{c_2} Y.
\end{equation*}
\begin{notation}
\label{not:delta}
Set $\calS_{\delta,c_2}=(k_0\otimes_D \R)/\frakn$.
\end{notation}
\begin{thm}
\label{thm:deltasubgroups}
Let $c_2\in \Hom_C(C[x],C)$ be such that $\frakh(c_2(x)+i)\neq 0$ for any $i\in \Z$ and $k_0\otimes_D \R$ is not the zero ring. Suppose that $\frakn$ is a maximal $\delta$-ideal of $k_0\otimes_D \R$. Then
$\stab(\frakn)$ is the $\delta$-Galois group of $\calS_{\delta,c_2}$ over $k_0$.
\end{thm}

Following Example~\ref{exam:sigmadeltagroups}, let us compute $\stab(\frakm)$ and $\stab(\frakn)$ in the following two examples respectively. Note that in Section 6.1 of \cite{Singer:algebraicandalgorithmic}, there is a simpler method to compute the $\sigma$-Galois group of $\sigma(Y)=A(m,c_1)Y$ over $C(m)$, \ie $\stab(\frakm)$. Here, to be consistent with the general case, we shall first compute $\frakm$ and then $\stab(\frakm)$.  
\begin{exam}
\label{exam:sigmagroups}
From Example~\ref{exam:sigmadeltagroups},
we have that 
$$
R=C(m,t)[X,1/\det(X)]/I
$$ 
where $I=\langle f_1,f_2,f_2\rangle$ with $f_i$ is given in (\ref{eqn:generators}). Set $\calD=C[t], \frakh=t^2-1, D=C[m,t,\frac{1}{\frakh}]$ and $\R=D[\calX,1/\det(\calX)]$ where $\calX=X \mod I$. Then one has that
\begin{equation}
\label{eqn:relations}
\calX_{11}\calX_{12}=1, \,\,\calX_{21}\calX_{22}=1,\,\,(\calX_{11}\calX_{22})^2-2t\calX_{11}\calX_{22}+1=0.
\end{equation}
Suppose that $c\in C=\Hom_C(C[t],t)$ such that $\frakh^c=c^2-1\neq 0$ and set $\alpha=c+\sqrt{c^2-1}$. Then $\alpha^2\neq 1$. As $P=\begin{pmatrix} 1 & 1 \\ \alpha^{-1} & \alpha \end{pmatrix}$ is a zero of $I^c$ in $\GL_2(C)$, $\langle I^c\rangle\neq C(m)[X,1/\det(X)]$. So $T=C(m)\otimes_D \R$ is not the zero ring. In the following, we separate two cases to compute a maximal $\sigma$-ideal $\frakm$ and $\stab(\frakm)$. For the sake of notation, we still use $\calX_{ij}$ to denote $1\otimes_D \calX_{ij}$. 

{\bf Case 1}: $\alpha$ is not a root of unity. 
Let $\frakm=(\calX_{11}\calX_{22}-\alpha)$. Since 
$$
  \sigma(\calX_{11}\calX_{22}-\alpha)=-\calX_{12}\calX_{21}+2c-\alpha=-\frac{1}{\alpha \calX_{11}\calX_{22}}(\calX_{11}\calX_{22}-\alpha).
$$
Thus $\frakm$ is a $\sigma$-ideal. Furthermore, $\frakm\neq T$, as $P$ is also a zero of $\calX_{11}\calX_{22}-\alpha$. Suppose that $f\in T\setminus \frakm$ and $\tilde{\frakm}$ is the $\sigma$-ideal generated by $\frakm$ and $f$. Using the relations (\ref{eqn:relations}), there are positive integers $\nu_1,\nu_2$ such that 
$$
  \calX_{22}^{\nu_2}\calX_{11}^{\nu_1}f=b(\calX_{11}\calX_{22}-\alpha)+\tilde{f}(\calX_{11})
$$ 
for some $b\in T,\tilde{f}\in C(m)[\calX_{11}]$. Since $f\notin \frakm$, $\tilde{f}\in \tilde{\frakm}\setminus \{0\}$. Let $\tilde{g}\in \tilde{\frakm}\setminus\{0\}$ be of the form $\tilde{g}=\calX_{11}^s+\sum_{i=0}^{s-1} a_i\calX_{11}^i$, where $a_i\in C(m)$ and $s$ is minimal. Then $\sigma(\tilde{g})=\calX_{21}^s+\sum_{i=0}^{s-1} \sigma(a_i)\calX_{21}^i\in \tilde{\frakm}$. Using the relations (\ref{eqn:relations}) again,
one sees that 
$$
  (\calX_{11}\calX_{22})^s\sigma(\tilde{g})=p(\calX_{11}\calX_{21}-\alpha)+\calX_{11}^s+\sum_{i=0}^{s-1} \sigma(a_i)\alpha^{s-i}\calX_{11}^i
$$ 
for some $p\in T$ and thus
$
   \sum_{i=0}^{s-1} (a_i-\alpha^{s-i}\sigma(a_i))\calX_{11}^i\in \tilde{\frakm}.
$
Since $s$ is minimal, we see that $a_i-\alpha^{s-i}\sigma(a_i)=0$ for all $i=0,\dots,s-1$. 
If $a_i=0$ for all $i$ then $\calX_{11}\in \tilde{\frakm}$ and thus $\tilde{\frakm}=T$, because $\calX_{11}$ is invertible. Now suppose that there is some $a_i$ that is not zero. Then from $\sigma(a_i)=\alpha^{s-i}a_i$, $\alpha^{s-i}=1$, a contradiction. Hence $\tilde{\frakm}=T$ and so $\frakm$ is a maximal $\sigma$-ideal. 

Now suppose $g=(g_{ij})\in G$. Then 
\[
   \rho_g(\calX_{11}\calX_{22}-\alpha)=\begin{cases}
               \frac{g_{11}}{\calX_{11}}(\calX_{11}\calX_{22}-\alpha) & g_{11}g_{22}=1\\
               -\frac{\alpha g_{12}}{\calX_{22}}(\calX_{11}\calX_{22}-\alpha^{-1}) & g_{12}g_{21}=1
   \end{cases}.
\]
Since $\alpha^2\neq 1$, $\calX_{11}\calX_{22}-\alpha^{-1}\notin \frakm$. So the condition $\rho_g(\frakm)\subset \frakm$ implies that $g_{11}g_{22}=1$. In other words, 
$$
  \stab(\frakm)=\left\{\begin{pmatrix} \xi & 0 \\ 0 & \xi^{-1}  \end{pmatrix} \mid  \xi\in C^{\times}\right\}.
$$ 

{\bf Case 2}: $\alpha$ is a $q$-th root of unity. Set $\frakm=(\calX_{11}\calX_{22}-\alpha,\calX_{11}^q-1)$. Then $\frakm$ is a nontrivial $\sigma$-ideal, as $P$ is a common zero of $\calX_{11}\calX_{22}-\alpha$ and $\calX_{11}^q-1$. Replacing $\frakm$ with $(\calX_{11}\calX_{22}-\alpha,\calX_{11}^q-1)$ in {\bf Case 1}, we have that $a_i-\alpha^{s-i}\sigma(a_i)=0$ for all $i=0,\dots,s-1$ and $s<q$. If all $a_i=0$ then $\calX_{11}\in \tilde{\frakm}$ and so $T=\tilde{\frakm}$. Otherwise there is some $a_i$ that is not zero. Then from $\sigma(a_i)-\alpha^{s-i}a_i=0$, $\alpha^{s-i}=1$, a contradiction with the fact that $\alpha$ is a $q$-th root of unity. Hence $\frakm$ is a maximal $\sigma$-ideal. Now suppose $g=(g_{ij})\in G$ such that $\rho_g(\frakm)\subset \frakm$. Then $g_{11}g_{22}=1$ by the argument as in {\bf Case 1}. So $g_{12}=0$ and then $\rho_g(\calX_{11}^q-1)=g_{11}^q\calX_{11}^q-1\in \frakm$. This implies that $g_{11}^q=1$. Consequently, one has that
$$
  \stab(\frakm)=\left\{\begin{pmatrix} \xi & 0 \\ 0 & \xi^{-1}  \end{pmatrix} \mid  \xi^q=1\right\}.
$$ 
\end{exam}
\begin{exam}
\label{exam:deltagroups}
Let $D,\frakh, \R$ be as in Example~\ref{exam:sigmagroups}. Let $c\in C=\Hom_C(C[m],C)$. As $f_i^c=f_i$ for all $i=1,2,3$, $\langle f_1,f_2,f_3\rangle\neq \langle 1 \rangle$. By Remark~\ref{rem:specializations}, $T=C(t)\otimes_{D} \R$ is not the zero ring. Let $\frakn$ be a maximal $\delta$-ideal of $T$. Denote $(\bar{\calX}_{ij})=1\otimes \calX \mod \frakn$. Suppose $\frakn\neq (0)$. Let $f\in \frakn\setminus\{0\}$. 
Using the relations (\ref{eqn:relations}), there are positive integers $\nu_1,\nu_2$ such that
$$
   \bar{\calX}_{11}^{\nu_1}\bar{\calX}_{22}^{\nu_2}\bar{f}=\sum_{i=0}^s (a_{i,1}\bar{\calX}_{11}\bar{\calX}_{22}+a_{i,0})\bar{\calX}_{11}^i=0.
$$
where $a_{i,0},a_{i,1}\in C(t)$. As $\bar{\calX}_{11}\bar{\calX}_{22}$ is algebraic over $C(t)$, $\bar{\calX}_{11}$ is algebraic over $C(t)$.
An easy calculation yields that
$$
   \delta(\bar{\calX}_{11})=\frac{c-1}{1-t^2}(t\bar{\calX}_{11}-\bar{\calX}_{21})=\frac{c-1}{1-t^2}\left(t-\frac{1}{\bar{\calX}_{11}\bar{\calX}_{22}}\right)\bar{\calX}_{11}.
$$
On the other hand, one has that
$$
   \delta(\bar{\calX}_{11}\bar{\calX}_{22})=\frac{-1}{1-t^2}\left(t-\frac{1}{\bar{\calX}_{11}\bar{\calX}_{22}}\right)\bar{\calX}_{11}\bar{\calX}_{22}.
$$
Therefore
$$
   \frac{\delta(\bar{\calX}_{11})}{\bar{\calX}_{11}}=(1-c)\frac{\delta(\bar{\calX}_{11}\bar{\calX}_{22})}{\bar{\calX}_{11}\bar{\calX}_{22}}.
$$
Since $\bar{\calX}_{11}\bar{\calX}_{22}$ is algebraic over $C(t)$, using the Puiseux series expansion of $\bar{\calX}_{11}\bar{\calX}_{22}$ at some pole, one sees that $\bar{\calX}_{11}$ is algebraic over $C(t)$ if and only if $c\in \Q$. This implies that $\frakn=(0)$ if and only if $c\notin \Q$. Now suppose $c=p/q$ with $p,q\in \Z,\gcd(p,q)=1$ and $q>0$. Then $\bar{\calX}_{11}^q=\beta(\bar{\calX}_{11}\bar{\calX}_{22})^{q-p}$ for some nonzero $\beta\in C$. Hence $\frakn=(\calX_{11}^q-\beta(\calX_{11}\calX_{22})^{q-p})$. In fact, $z^q-\beta(\calX_{11}\calX_{22})^{q-p}$ is the minimal polynomial of $\bar{\calX}_{11}$ over $C(t,\calX_{11}\calX_{22})$. As $\calX_{12}\calX_{11}=1$ and $\calX_{21}\calX_{22}=1$, one also has that $\frakn=(\calX_{12}^q-\beta(\calX_{12}\calX_{21})^{p-q})$. Therefore, for any $(g_{ij})\in G$,
\[
\rho_g(\calX_{12}^p-\calX_{22}^{p-q})=
\begin{cases}
   g_{11}^q\calX_{11}^q-\beta(\calX_{11}\calX_{22})^{q-p} & g_{12}=g_{21}=0\\
   g_{21}^q\calX_{12}^q-\beta(\calX_{12}\calX_{21})^{q-p} & g_{11}=g_{22}=0
\end{cases} 
\]
and the condition $\rho_g(\calX_{11}^q-\beta(\calX_{11}\calX_{22})^{p-q})\in \frakn$ implies that either $g_{11}^q=1$ or $g_{21}^q=1$. These imply that
\begin{align*}
    \stab(\frakn)=
    \begin{cases}
       G & c\notin\Q \\
       \left\{\begin{pmatrix} \xi & 0 \\ 0 & \xi^{-1} \end{pmatrix}, \begin{pmatrix} 0 & \xi \\ \xi^{-1} & 0 \end{pmatrix} \mid \xi^q=1\right\} & c=\frac{p}{q}\in \Q.
    \end{cases}
\end{align*}
\end{exam}

\subsection{$G$ is the product of two suitable algebraic subgroups}
\label{subsec:product}
In this subsection, we shall show that for suitable $c_1$ and $c_2$ we have that $G$ is the product of $\stab(\frakm)$ and $\stab(\frakn)$, where $\frakm$ is any maximal $\sigma$-ideal of $\calS_{\sigma,c_1}$, and $\frakn$ is any maximal $\delta$-ideal of $\calS_{\delta,c_2}$.

Throughout this subsection, let $\breve{R}$ and $\K$ be as in Notation~\ref{not:sufficientcondition}. Proposition~\ref{prop:strongsigmapv} implies that $\breve{R}$ is a $\sigma$-Picard Vessiot ring over $\K^\sigma(x)$ for $\sigma(Y)=AY$. Remark that for each $\tau\in \sigma\delta$-$\Gal(\K/k_0(x))$ and $a\in \K^\sigma$ one has that $\tau(a)\in \K^\sigma$. Therefore $\K^\sigma(x)$ is invariant under the action of $\sigma\delta$-$\Gal(\K/k_0(x))$. From this, $\sigma\delta$-$\Gal(\K/\K^\sigma(x))$ is a normal subgroup of $\sigma\delta$-$\Gal(\K/k_0(x))$. Using the normality of $\sigma\delta$-$\Gal(\K/\K^\sigma(x))$, we shall show that $\K^\sigma(x)$ is the field of fractions of $\R^{\sdgal(\K/\K^\sigma(x))}$. To this end, we need the following lemma which was shown in the proof of Proposition 6.3.5 on page 157 of \cite{CrespoandHajtodifferentialgaloistheory}. Suppose $H$ is an algebraic subgroup of $\GL_n(C)$. Then $C[H]$ can be endowed with an $H$-module structure by setting $h(a)=a(\calZ h)$ for any $a\in C[H]$ and any $h\in H$.
\begin{lem}
\label{lm:character}
Let $H$ be an affine algebraic group over $C$. Let $N$ be a normal algebraic subgroup of $H$ and $\chi$ a character of $N$. Then there is a nonzero $a\in C[H]$ such that $h(a)=\chi(h)a$ for any $h\in N$.
\end{lem}
\begin{proof}
See the proof of Proposition 6.3.5 on page 157 of \cite{CrespoandHajtodifferentialgaloistheory}.
\end{proof}
\begin{lem}
\label{lm:constants}
Let $P,Q\in \K^\sigma[x]$ be such that $\gcd(P,Q)=1$ and $Q$ is monic. Suppose that $P/Q\in \R$. Then there is a nonzero $r\in \K^\sigma[x]$ such that $rQ\in \calD[x]$ and all coefficients of $rP$ are in $\R^\sigma$.
\end{lem}
\begin{proof}
We first show that there is a nonzero $r\in \K^\sigma[x]$ such that $rQ\in \calD[x]$. Suppose $Q=\prod_{i=1}^\ell Q_i$ where $Q_i$ is monic and irreducible over $\K^\sigma[x]$. Let $s_i$ be the largest integer such that $Q_i(x+s_i)$ divides $Q$. Since $P/Q\in \R$, the set $\{ \sigma^i(P/Q) \mid \forall i\geq 0\}$ generates a $k_0(x)$-vector space of finite dimension. Hence there are $a_0, \dots, a_m\in \calD[x]$ with $a_m\neq 0$ such that $\sum_{i=0}^m a_i \sigma^i(P/Q)=0$. Multiplying both sides by $\prod_{i=0}^m \sigma^i(Q)$ yields that
$$
    \left(\prod_{i=0}^{m-1}\sigma^i(Q)\right)\sigma^m(P)a_m=\sigma^m(Q)N,\,\,N\in \K^\sigma[x].
$$
If $Q_i(x+s_i+m)$ divides $\sigma^j(Q)$ for some $0\leq j \leq m-1$ then $Q_i(x+s_i+m-j)$ divides $Q$. This contradicts with the choice of $s_i$. Hence $Q_i(x+s_i+m)$ does not divide $\prod_{j=0}^{m-1}\sigma^j(Q)$. It is clear that $Q_i(x+s_i+m)$ does not divide $\sigma^m(P)$ too. While $Q_i(x+s_i+m)$ divides $\sigma^m(Q)$. This implies that $Q_i(x+s_i+m)$ divides $a_m$. Write $a_m=Q_i(x+s_i+m)r_i(x+m+s_i)$ for some $r_i\in \sigma^\sigma[x]$. Then $r_iQ_i=a_m(x-s_i-m)\in \calD[x]$. Set $r=\prod_{i=1}^\ell r_i$. Then $rQ=\prod_{i=1}^\ell a_m(x-s_i-m)\in \calD[x]$.

It is clear that $rP\in \R$. Write $rP=\sum_{i=0}^\ell p_i x^i$. Applying $\sigma^i, i=0,\dots, \ell$ to both sides yields that $M(p_0,\dots,p_\ell)^t=(rP,\dots,\sigma^\ell(rP))$. Here $M$ is the Vandermonde matrix formed by $x,x+1,\dots,x+\ell$, whose inverse has entries in $\R$. Therefore $(p_0,\dots,p_\ell)^t=M^{-1}(rP,\dots,\sigma^\ell(rP))^t \in \R^{\ell+1}$. Consequently, $p_i\in \R$ for all $i$. Since $p_i\in \K^\sigma$, $p_i\in \R^\sigma$.
\end{proof}
\begin{prop}
\label{prop:fieldoffraction}
Let $H=\sigma\delta$-$\Gal(\K/\K^\sigma(x))$. Then
\begin{enumerate}
\item $\K^\sigma(x)$ is the field of fractions of $\R^H$,
\item $\K^\sigma$ is the field of fractions of $\R^\sigma$.
\end{enumerate}
\end{prop}
\begin{proof}
1. Suppose $f\in \K^\sigma(x)$. Set $U=\{a\in \R \mid af\in \R\}$. We shall show that $U\cap \R^H\neq \{0\}$. Let $a\in U\setminus\{0\}$. Then $\{h(a)\mid h\in H\}$ generates a $C$-vector space of finite dimension. Suppose that $\{a_1,\dots,a_m\}$ is a basis of this vector space. Since $\R$ is invariant under the action of $H$, this vector space is a subspace of $\R$. Moreover, as $\R$ is $\sigma\delta$-simple, by Proposition~\ref{prop:lineardependence}, there are $\theta_1,\dots,\theta_m\in \Theta$ with $\theta_1=1$ such that $d=\det((\theta_i(a_j))_{1\leq i,j\leq m})\neq 0$. Since $a_i\in U$ and $\theta_1=1$, $d\in U$. For each $h\in H$, $h(d)=\chi(h)d$ where $\chi$ is a character of $H$. We need to find an element $\tilde{a}\in \R$ such that $h(\tilde{a})=\chi^{-1}(h)\tilde{a}$ for any $h\in H$. Once we have such $\tilde{a}$, $\tilde{a}d\in \R^H\cap U$ and thus $\tilde{a}df\in \R^H$. Corollary~\ref{cor:groupinvariantideal} implies that neither $\tilde{a}$ nor $d$ is a zero divisor of $\breve{R}$ and thus none of them is a zero divisor of $\R$. This implies that $\tilde{a}d\in \R^H$ is not a zero divisor and so $f=r/\tilde{a}d$ for some $r\in \R^H$. We first find a nonzero $\beta\in R$ such that $h(\beta)=\chi^{-1}(h)\beta$. By Lemma~\ref{lm:character}, there is a nonzero $b\in C[G]$ such that $h(b)=\chi^{-1}(h)b$ for any $h\in H$. Define an action of $G$ on $R\otimes_{k_0(x)} R$ and $R\otimes_C C[G]$ by $g(r_1\otimes r_2)=r_1\otimes g(r_2)$ and $g(r_1 \otimes r_2)=r_1\otimes g(r_2)$ respectively. Then both $R\otimes_{k_0(x)} R$ and $R\otimes_C C[G]$ become $G$-modules and it is easy to see that the isomorphism $\varphi: R\otimes_{k_0(x)} R\rightarrow R\otimes_C C[G]$ given by $\varphi(r_1\otimes r_2)=(r_1\otimes 1)r_2(\calX\otimes \calZ)$ is a $G$-module isomorphism. Write $\varphi^{-1}(b)=b(\calX^{-1}\otimes \calX)=\sum_{i=1}^s \alpha_i\otimes \beta_i$ where $\{\alpha_i\}\subset R$ is linearly independent over $k_0(x)$ and none of $\beta_i$ is zero. Suppose $h\in H$. Then $h(\varphi^{-1}(b))=\sum_{i=1}^s \alpha_i\otimes h(\beta_i)$. On the other hand, one has that $h(\varphi^{-1}(b))=\varphi^{-1}(h(b))=\varphi^{-1}(\chi^{-1}(h)b)=\sum_{i=1}^s \alpha_i\otimes \chi^{-1}(h)\beta_i$. Therefore for each $i=1,\dots,s$, $h(\beta_i)=\chi^{-1}(h)\beta_i$ for any $h\in H$. Thus we can take $\beta$ to be any $\beta_i$. Let $q\in \calD[x]\setminus\{0\}$ be such that $p\beta\in \R$. It is easy to verify that $h(p\beta)=ph(\beta)=\chi^{-1}(h)p\beta$. So $p\beta$ satisfies the requirement.

2. Suppose that $f\in \K^\sigma$. The previous result implies that $f=a/b$ with $a,b\in \R^H\subset \K^\sigma(x)$. Write $a=P_1/Q_1,b=P_2/Q_2$ where $P_i,Q_i\in \K^\sigma[x], \gcd(P_i,Q_i)=1$ and $Q_i$ is monic. Due to Lemma~\ref{lm:constants}, there are nonzero $r_i\in \K^\sigma[x]$ such that all coefficients of $r_iP_i, r_iQ_i$ are in $\R$. From $fP_2Q_1=P_1Q_2$, one sees that $fr_2P_2 r_1Q_1=r_1P_1 r_2Q_2$  and thus $\lc(r_2P_2r_1Q_1)f=\lc(r_1P_1r_2Q_2)$, where $\lc(\cdot)$ denotes the leading coefficient of a polynomial. As $\lc(r_iP_i), \lc(r_iQ_i)\in \R\cap \K^\sigma=\R^\sigma$, $f$ is in the field of fractions of $\R^\sigma$.
\end{proof}

\begin{prop}
\label{prop:differentialgaloispart}
Let $\frakn$ be as in Theorem~\ref{thm:deltasubgroups} and $\calS_{\delta,c_2}=(k_0\otimes_D \R)/\frakn$.
The $\calD$-$\delta$-homomorphism $\phi: \R^\sigma\rightarrow \calS_{\delta,c_2}$ given by $a(\calX) \mapsto a^{c_2}(\bar{\calX})$ is injective, where $\bar{\calX}=1\otimes \calX \mod \frakn$. Furthermore, $\phi(\R^\sigma)$ is invariant under the action of $\dgal(\calS_{\delta,c_2}/k_0)$.
\end{prop}
\begin{proof}
 Note that $\phi(a)=a^{c_2}(\bar{\calX})=\overline{1\otimes a}$ for any $a\in \R^\sigma$. To prove the injectivity of $\phi$, it suffices to show that $1\otimes a\notin \frakn$ if $a\neq 0$. Suppose that $a\in \R^\sigma\setminus \{0\}$. Since $\sigma(a)=a$ and $\R$ is $\sigma\delta$-simple, there are $b_1,\dots,b_s\in \R$ such that
$1=\sum_{i=1}^s b_i\delta^i(a)$. This implies that $1\otimes 1=\sum_{i=1}^s 1\otimes b_i\delta^i(a)=\sum_{i=1}^s \delta^i(1\otimes a)(1\otimes b_i)$. In other words, the $\delta$-ideal generated by $1\otimes a$ equals $k_0\otimes_D \R$. Hence $1\otimes a\notin \frakn$.

Suppose that $\tau\in \dgal(\calS_{\delta,c_2}/k_0)$ and $a\in \R^\sigma$. Using the identification given in Proposition~\ref{prop:galoisgroups} with $R=\calS_{\delta,c_2}$, there is $\bar{\gamma}\in \Hom_C(C[G]/\Phi_{\frakn,\frakn}, C)$ such that $\bar{\gamma}(\bar{\calZ})=\bar{\calX}^{-1}\tau(\bar{\calX})$, where $\bar{\calZ}=\calZ \mod \Phi_{\frakn,\frakn}$. Let $\gamma$ be the unique element in $\Hom_C(C[G],C)$ such that $\gamma(\calZ)=\bar{\gamma}(\bar{\calZ})$. Then $\tau(\phi(a))=a^{c_2}(\bar{\calX}\bar{\gamma}(\bar{\calZ}))=a^{c_2}(\bar{\calX}\gamma(\calZ))$. Since $\gamma\in \Hom_C(C[G],C)$ and $\R^\sigma$ is invariant under the action of $\sdgal(R/k)$, $a(\calX\gamma(\calZ))\in \R^\sigma$. This implies that
$$
   \tau(\phi(a))=a^{c_2}(\bar{\calX}\gamma(\calZ))=\phi(a(\calX \gamma(\calZ)))\in \phi(\R^\sigma).
$$
Thus $\phi(\R^\sigma)$ is invariant under the action of $\dgal(\calS_{\delta,c_2}/k_0)$.
\end{proof}

Let $\F_{\delta,c_2}$ be the field of fractions of $\calS_{\delta,c_2}$. Then the $\delta$-homomorphism $\phi$ given in Proposition~\ref{prop:differentialgaloispart} can be extended into an embedding of $\K^\sigma$ into $\F_{\delta,c_2}$ and $\phi(\K^\sigma)$ is invariant under the action of $\dgal(F_{\delta,c_2}/k_0)$ that is an algebraic subgroup of $G$. By the Galois correspondence (see for example Proposition 1.34 on page 25 of \cite{van2012galois}), $\phi(\K^\sigma)$ is a $\delta$-Picard-Vessiot field for some $\delta$-linear system over $k_0$ and the following canonical map
\begin{align*}
   \pi: \dgal(\F_{\delta,c_2}/k_0) &\longrightarrow  \dgal(\phi(\K^\sigma)/k_0) \\
          g &\longmapsto g|_{\phi(\K^\sigma)}.
\end{align*}
is surjective and has kernel $\dgal(\F_{\delta,c_2}/\phi(\K^\sigma))$.
\begin{prop}
\label{prop:generators}
Let $\sdgal(\K/\K^\sigma(x))$ be identified with an algebraic subgroup $H$ of $G$ and $H'=\stab(\frakn)$, where $\frakn$ is as in Theorem~\ref{thm:deltasubgroups}.
Then $G=H H'$.
\end{prop}
\begin{proof}
Recall that for $g\in G=\Hom_C(C[G],C)$ the corresponding automorphism of $R$ over $k$ is $\tau_g\in \sdgal(R/k)$ such that $g(\calZ)=\calX^{-1}\tau_g(\calX)$. Furthermore if $g\in \stab(\frakn)=\Hom_C(C[G]/\Phi_{\frakn,\frakn}, C)\subset \Hom_C(C[G],C)$ then the corresponding automorphism of $\calS_{\delta,c_2}$ over $k_0$ is $\gamma_g\in \dgal(\calS_{\delta,c_2}/k_0)$ such that $\gamma_g(\bar{\calX})=\bar{\calX}g(\calZ)$ where $\bar{\calX}= 1\otimes \calX \mod \frakn$. Now suppose $g\in G$. Then $\tau_g\in \sdgal(\K/k)$ and $\tau_g |_{\K^\sigma}\in \dgal(\K^\sigma/k_0)$. Let $\phi$ be given as in Proposition~\ref{prop:differentialgaloispart}. Then $\phi$ can be viewed as a $k_0$-$\delta$-isomorphism from $\K^\sigma$ to $\phi(\K^\sigma)$ and thus $\phi\circ \tau_g|_{\K^\sigma} \circ\phi^{-1} \in \dgal(\phi(\K^\sigma)/k_0)$. Let $h\in\stab(\frakn)$ be such that $\gamma_h=\phi\circ \tau_g|_{\K^\sigma} \circ\phi^{-1}$, \ie $\phi\circ \tau_g|_{\K^\sigma} \circ\phi^{-1}\circ \gamma_h^{-1}=\id$. We claim that $\tau_{gh^{-1}}(f)=f$ for all $f\in \K^\sigma(x)$. It suffices to show that $\tau_{gh^{-1}}(f)=f$ for all $f\in \K^\sigma$. Suppose $f\in \K^\sigma$. Due to Proposition~\ref{prop:fieldoffraction}, we may write $f=a/b$ with $a,b\in \R^\sigma$ and $b\neq 0$. We then have that
\begin{align*}
  \phi\left(\tau_{gh^{-1}}\left(\frac{a}{b}\right)\right)&=\phi\left(\frac{a(\calX g(\calZ)h(\calZ)^{-1})}{b(\calX g(\calZ)h(\calZ)^{-1})}\right)=\frac{a^c(\bar{\calX} g(\calZ)h(\calZ)^{-1})}{b^c(\bar{\calX} g(\calZ)h(\calZ)^{-1})}\\
  &=\phi\circ \tau_g|_{\K^\sigma} \circ\phi^{-1}\circ \gamma_h^{-1}\left(\frac{a^c(\bar{\calX})}{b^c(\bar{\calX})}\right)=\frac{a^c(\bar{\calX})}{b^c(\bar{\calX})}=\phi\left(\frac{a}{b}\right).
\end{align*}
Since $\phi$ is injective, $\tau_{gh^{-1}}(f)=f$. This proves our claim. Hence one has that $\tau_{gh^{-1}}\in \sdgal(\K/\K^\sigma(x))$ and then $gh^{-1}\in H$ under the identification. Therefore $g\in HH'$. It is clear that $HH'\subset G$. So $G=HH'$.
\end{proof}

\begin{exam}
\label{exam:normalsubgroups}
In Examples~\ref{exam:sigmadeltagroups} and~\ref{exam:deltagroups}, we have already known that
\begin{align*}
  G&=\{ (g_{ij})\in\GL_2(C) \mid g_{11}g_{12}=0, g_{21}g_{22}=0, g_{11}g_{22}+g_{12}g_{21}=1\}\\
  &=\left\{\begin{pmatrix} g_{11} & 0 \\ 0 & g_{22} \end{pmatrix} \mid g_{11}g_{22}=1\right\}\cup \left\{\begin{pmatrix} 0 & g_{12} \\ g_{21} & 0\end{pmatrix} \mid g_{12}g_{21}=1\right\}
\end{align*}
and
$$
   H'=
    \begin{cases}
       G & c\notin\Q \\
       \left\{\begin{pmatrix} \xi & 0 \\ 0 & \xi^{-1} \end{pmatrix}, \begin{pmatrix} 0 & \xi \\ \xi^{-1} & 0 \end{pmatrix} \mid \xi^q=1\right\} & c=\frac{p}{q}\in \Q.
    \end{cases}
$$
Now let us compute $H=\sdgal(\K/\K^\sigma(x))$. We first compute $\K^\sigma$. It is clear that $\K^\sigma\subseteq K=C(t,\sqrt{t^2-1})$. On the other hand, $(\calX_{11}\calX_{22})^2-2t\calX_{11}\calX_{22}+1=0$ and $\calX_{11}\calX_{22}\in \R^\sigma$. As $[C(t,\calX_{11}\calX_{22}):C(t)]=2$, one has that $\R^\sigma=\K^\sigma=K=C(t)(\calX_{11}\calX_{22})$. Next, we compute $\sdgal(\K/K(m))$. Note that for any $g=(g_{ij})\in \sdgal(\K/C(m,t))$, $g\in \sdgal(\K/K(m))$ if and only if $g(\calX_{11}\calX_{22})=\calX_{11}\calX_{22}$. An easy calculation yields that if $g_{12}g_{21}=1$ then $\rho_g(\calX_{11}\calX_{22})=\calX_{12}\calX_{21}$. Since $(\calX_{11}\calX_{22})^2\neq 1$, using the relations (\ref{eqn:relations}), $\calX_{11}\calX_{22}\neq \calX_{12}\calX_{21}$. Hence one has that $g\in \sdgal(\K/K(x))$ if and only if $g_{11}g_{22}=1$ and $g_{12}g_{21}=0$. Hence
$$
    H=\{(g_{ij})\in G \mid g_{11}g_{22}=1\}=\left\{\begin{pmatrix} g_{11} & 0 \\ 0 & g_{22} \end{pmatrix} \mid g_{11}g_{22}=1\right\}.
$$
It is easy to see that $G=HH'$.
\end{exam}

In the remainder of this subsection, we shall prove that $H=\stab(\frakm)$ for any maximal $\sigma$-ideal $\frakm$ of $\calS_{\sigma,c_1}$ with suitable $c_1$. 
Since $\K^\sigma$ is a $\delta$-Picard-Vessiot field for some $\delta$-linear system over $k_0$, there is $\eta\in \GL_m(\K^\sigma)$ for some $m$ such that $\K^\sigma=k_0(\eta)$ and moreover $k_0[\eta,1/\det(\eta)]$ is a $\delta$-Picard-Vessiot ring over $k_0$ for the corresponding $\delta$-linear system. In particular, $k_0[\eta,1/\det(\eta)]$ is $\delta$-simple. 
\begin{notation}
\label{not:assumption2}
We assume that
\begin{itemize}
\item $\tilde{\calD}=\tilde{\calD}[\eta,1/\det(\eta)]$;
\item $\tilde{D}=\tilde{\calD}[x][\{\frac{1}{\sigma^i(\frakh)} \mid \forall i\in \Z\}]$;
\item $\tilde{\R}=\tilde{D}[\calX,1/\det(\calX)]$.
\end{itemize}
\end{notation}
We see that $\tilde{\calD}$ is finitely generated over $C$, because $\calD$ is. Using an argument similar to the proof of Lemma~\ref{lm:sigmadeltasimple}, one sees that $\tilde{\calD}$ is $\delta$-simple and $\tilde{D}$ is $\sigma\delta$-simple. Let $\tilde{c}_1: \tilde{\calD}\rightarrow C$ be a $C$-homomorphism that uniquely lifts to a $C[x]$-homomorphism from $\tilde{D}$ to $C(x)$. As before, we consider $\tilde{T}=C(x)\otimes_{\tilde{D}} \tilde{\R}$. Then if $\tilde{T}$ is not the zero ring then it is a $\sigma$-ring. Moreover, by Proposition~\ref{prop:extensiontorsors} with $D=\tilde{D}, T=\tilde{T}$ and $G=H=\Hom_C((\breve{R}\otimes_{\K^\sigma(x)}\breve{R})^{\sigma\delta}, C)$, one has the following $\tilde{T}$-isomorphism:
\begin{align}
\label{eqn:torsors}
    \tilde{T}\otimes_{C(x)} \tilde{T} &\longrightarrow \tilde{T}\otimes_C C[H]\\
    a\otimes b &\longmapsto (a\otimes 1)b(\tilde{\calX} \otimes \calZ) \notag
\end{align}
where $\tilde{\calX}=1\otimes_{\tilde{D}} \calX$.
\begin{lem}
\label{lm:sigmasimple}
Let $H$ be as in Proposition~\ref{prop:generators}. 
There is a Zariski dense subset $U_1$ of $\Hom_C(\tilde{\calD},C)$ such that for any $c_1\in U_1$ if $C(x)\otimes_{\tilde{D}} \tilde{\R}$ is not the zero ring then it is $\sigma$-simple and $H$ is the $\sigma$-Galois group of $C(x)\otimes_{\tilde{D}} \tilde{\R}$ over $C(x)$.
\end{lem}
\begin{proof}
Due to Proposition 2.4 of \cite{chatzidakis2007definitions}, $\bar{T}=\overline{\K^\sigma}(x)\otimes_{\K^\sigma(x)} \breve{R}$ is a $\sigma$-Picard-Vessiot ring over $\overline{\K^\sigma}(x)$ for $\sigma(Y)=AY$. By Lemma~\ref{lm:connections}, $H(\K^\sigma)$ is the $\sigma$-Galois group of $\sigma(Y)=AY$ over $\K^\sigma(x)$. By Corollary 2.5 of \cite{chatzidakis2007definitions} or using an argument similar to the proof of Proposition~\ref{prop:strongsigmapv}, one sees that $H(\overline{\K^\sigma})$ is the $\sigma$-Galois group of $\sigma(Y)=AY$ over $\overline{\K^\sigma}(x)$. Embedding $H$ into $\GL_n(C)$, we consider $H$ as an algebraic subgroup of $\GL_n(C)$ for the monent. Let $S\subset C[Z,1/\det(Z)]$ be a finite set that generates the vanishing ideal of $H$. Then $S$ also generates the vanishing ideal of $H(\overline{\K^\sigma})$. Due to Theorem 1.2 of \cite{feng2017difference} with $\mathbb{X}=\Hom_C(\tilde{\calD},C)$, there is a Zariski dense subset $U_1$ of $\Hom_C(\tilde{\calD},C)$ such that for any $c_1\in U_1$, the variety in $\GL_n(C)$ defined by $S^{c_1}$ is the $\sigma$-Galois group of $\sigma(Y)=A^{c_1}Y$ over $C(x)$. Since $S=S^{c_1}$, $H$ is the $\sigma$-Galois group of $\sigma(Y)=A^{c_1}Y$ over $C(x)$ for any $c_1\in U_1$.

Suppose $c_1\in U_1$ and $\tilde{T}=C(x)\otimes_{\tilde{D}} \tilde{\R}$ is not the zero ring. We shall show that $\tilde{T}$ is $\sigma$-simple and then it is a $\sigma$-Picard-Vessiot ring over $C(x)$ for $\sigma(Y)=A^{c_1}Y$. Let $\tilde{\frakm}$ be a maximal $\sigma$-ideal of $\tilde{T}$. Due to 
Theorem~\ref{thm:sigmasubgroups} with $\frakm=\tilde{\frakm}$ and $G=H$, $\stab(\tilde{\frakm})=\{h\in H \mid \rho_h(\tilde{\frakm})\subset \tilde{\frakm}\}$ is the $\sigma$-Galois group of $\tilde{T}/\tilde{\frakm}$ over $C(x)$. In other words, $\stab(\tilde{\frakm})$ is also the $\sigma$-Galois group of $\sigma(Y)=A^{c_1}Y$ over $C(x)$. Hence there is a $g\in \Hom_C(C[\GL_n], C)$ such that $\stab(\tilde{\frakm})=gHg^{-1}$. Since $\stab(\tilde{\frakm})\subset H$, one has that $gHg^{-1}\subset H$ and thus $H\supset gHg^{-1}\supset g^2 H g^{-2}\supset \dots$. The Noetherian property of $H$ then implies that $H=gHg^{-1}=\stab(\tilde{\frakm})$. This could happen only if $\tilde{\frakm}=(0)$ by (\ref{eqn:torsors}). So $\tilde{T}$ is $\sigma$-simple.
\end{proof}

\begin{prop}
\label{prop:differencespecialization}
Let $H$ be as in Proposition~\ref{prop:generators}. Then there is a Zariski dense subset $U_1$ of $\Hom_C(\calD,C)$ such that $H=\stab(\frakm)$, where $\frakm$ is any maximal $\sigma$-ideal of $C(x)\otimes_D \R$.
\end{prop}
\begin{proof}
Note that $\calD\subset \tilde{\calD}, D\subset \tilde{D}$ and $\R\subset \tilde{\R}$. By Lemma~\ref{lm:sigmasimple}, there is a Zariski dense subset $\tilde{U}_1$ of $\Hom_C(\tilde{\calD},C)$ such that for any $\tilde{c}_1\in \tilde{U}_1$ if $\tilde{T}=C(x)\otimes_{\tilde{D}} \tilde{\R}$ is not the zero ring then $\tilde{T}$ is a $\sigma$-Picard-Vessiot ring over $C(x)$ for $\sigma(Y)=A^{\tilde{c}_1}Y$ and $H$ is the $\sigma$-Galois group of $\tilde{T}$ over $C(x)$. By Lemma~\ref{lm:zerorings}, there is a nonzero element in $\tilde{D}$, say $a$, such that for any $\tilde{c}\in \Hom_{C[x]}(\tilde{D},C(x))$ if $\tilde{c}(a)\neq 0$ then $C(x)\otimes_{\tilde{D}} \tilde{\R}$ is not the zero ring. Write $a=a_1/a_2$ with $a_i\in \tilde{\calD}[x]$.
Let 
$$
U_1=\{\tilde{c}_1|_{\calD} \mid \tilde{c}_1\in \tilde{U}_1 \,\,\mbox{and}\,\,\tilde{c}_1(\lc(a_1a_2\frakh))\neq 0\}
$$ 
where $\lc(\cdot)$ denotes the leading coefficient of a polynomial in $x$.
We shall prove that $U_1$ has the desired property. It is clear that $U_1$ is a Zariski dense subset of $\Hom_C(\calD,C)$. Suppose $c_1\in U_1$ and $\tilde{c}_1\in \tilde{U}_1$ such that $c_1=\tilde{c}_1|_D$. Then $\tilde{T}=C(x)\otimes_{\tilde{D}} \tilde{\R}$ is not the zero ring. We have the natural homomorphisms $\psi: T\rightarrow \tilde{T}, 1\otimes_D \calX \mapsto 1\otimes_{\tilde{D}} \calX$. Thus $T$ is not the zero ring too. Set $\frakm=\ker(\psi)$. Then $\frakm$ is a maximal $\sigma$-ideal. Suppose $h\in H\subset \Hom_C(C[G],C)$. Then $h$ induces a $\sigma$-$C(x)$-automorphism of $T$ defined as $a\otimes_D b(\calX)\mapsto a\otimes_D b(\calX h(\calZ))$ and also induces a $\sigma$-$C(x)$-automorphism of $\tilde{T}$ defined as $a\otimes_{\tilde{D}} b(\calX)\mapsto a\otimes_{\tilde{D}} b(\calX h(\calZ))$. Suppose $\sum a_i\otimes_D b_i\in \frakm$, \ie $\sum a_i\otimes_{\tilde{D}} b_i=0$. Then $\sum a_i \otimes_{\tilde{D}} b_i(\calX h(\calZ))=0$ by Lemma~\ref{lm:sigmasimple}. This implies that $\sum a_i\otimes_D b_i(\calX h(\calZ))\in \frakm$. Hence $h\in \stab(\frakm)$ and thus $H\subset \stab(\frakm)$. As both $H$ and $\stab(\frakm)$ are the $\sigma$-Galois group of $\sigma(Y)=A^{c_1}Y$ over $C(x)$, $H=\stab(\frakm)$ by an argument similar to that in Lemma~\ref{lm:sigmasimple}. Finally, suppose that $\frakm'$ is another maximal $\sigma$-ideal of $T$. Then there is a $g\in G$ such that $\stab(\frakm')=g\stab(\frakm)g^{-1}$ by Corollary~\ref{cor:maximaltomaximal}. Since $H$ is a normal subgroup of $G$, $\stab(\frakm')=gHg^{-1}=H$.
\end{proof}

\begin{thm}
\label{thm:mainresult}
There is a Zariski dense subset $U_1$ of  $\Hom_C(\calD,C)$ and a Zariski dense subset $U_2$ of $\Hom_C(C[x],C)$ such that for any $c_1\in U_1$ and any $c_2\in U_2$, $G=\stab(\frakm)\stab(\frakn)$, where $\frakm, \frakn$ are given as in Theorem~\ref{thm:sigmasubgroups} and Theorem~\ref{thm:deltasubgroups} respectively.
\end{thm}
\begin{proof}
Let $b$ be a nonzero element in $D$ such that for any $c\in \Hom_C(D,k_0)$ with $c(b)\neq 0$, $k_0\otimes_D \R$ is not the zero ring. Write $b=b_1/b_2$ where $b_1,b_2\in \calD[x]$ and $b_2\neq 0$ and set
$$
   U_2=\{ c_2\in \Hom_C(C[x],C) \mid b_1(c_2(x))b_2(c_2(x))\frakh(c_2(x)+i)\neq 0\,\,\forall i\in \Z\}.
$$
We claim that $U_2$ is Zariski dense. Otherwise assume that $U_2=\{u_1,\dots,u_\ell\}$. Let $u_{\ell+1},\dots,u_\nu$ be all zeroes of $b_1b_2$ in $C$  and $v_1,\dots,v_s$ are all zeroes of $\frakh$ in $C$. Then
$$
C=\{u_i \mid 1\leq i \leq \nu\}\cup \cup_{i=1}^s \{v_i+\Z\}\subseteq \Z+\sum_{i=1}^\nu \Z u_i +\sum_{i=1}^s \Z v_i\subseteq C.
$$
This implies that $C$ is a finitely generated $\Z$-module. However this is impossible, since $\Q\subset C$ is not finitely generated as a $\Z$-module. Due to Proposition~\ref{prop:differencespecialization}, let $U_1$ be a Zariski dense subset of $\Hom_C(\calD,C)$ such that for any $c_1\in U_1$, $H=\stab(\frakm)$. The theorem then follows from Proposition~\ref{prop:generators}.
\end{proof}
\begin{exam}
 Let $U_1=\{c\in C \mid \mbox{$c+\sqrt{c^2-1}$ is not a root of unit}\}$ and $U_2=C$. From Examples~\ref{exam:sigmagroups}, \ref{exam:deltagroups} and~\ref{exam:normalsubgroups}, one has that $G=\stab(\frakm)\stab(\frakn)$ for any $c_1\in U_1$ and any $c_2\in U_2$.
\end{exam}



\bibliographystyle{IEEEtran}
\bibliography{bibfile}

\end{document}